\documentclass[11pt,a4paper]{amsart}
\usepackage[margin=1.1in]{geometry}
\usepackage{amsmath,amssymb,amsthm,mathtools}
\usepackage{graphicx}
\usepackage{mathrsfs}
\usepackage{xcolor}
\definecolor{linkblue}{RGB}{25,60,120}
\usepackage[colorlinks=true,linkcolor=linkblue,citecolor=linkblue,urlcolor=linkblue]{hyperref}
\usepackage{microtype}

\newtheorem{theorem}{Theorem}[section]
\newtheorem{proposition}[theorem]{Proposition}
\newtheorem{lemma}[theorem]{Lemma}
\newtheorem{corollary}[theorem]{Corollary}
\theoremstyle{definition}
\newtheorem{definition}[theorem]{Definition}
\newtheorem{assumption}[theorem]{Assumption}
\theoremstyle{remark}
\newtheorem{remark}[theorem]{Remark}

\newcommand{\R}{\mathbb{R}}

\newcommand{\Id}{\mathrm{Id}}

\newcommand{\supp}{\operatorname{supp}}
\newcommand{\Res}{\operatornamewithlimits{Res}}
\newcommand{\dd}{\,\mathrm{d}}
\newcommand{\ellf}{\ell}
\newcommand{\sA}{\mathscr{A}}
\newcommand{\sG}{\mathcal{G}}
\newcommand{\sX}{\mathcal{X}}

\title[Plasmon poles for the linearized Coulomb--Hartree--Fock equation]{Persistence and analyticity of plasmon poles for the linearized Coulomb--Hartree--Fock equation in the small-exchange regime}

\author{Yuya Dan}
\address{Matsuyama University, 4-2 Bunkyo-cho, Matsuyama, Ehime 790-8578, Japan}
\email{dan@g.matsuyama-u.ac.jp}

\date{August 30, 2026}


\begin{document}

\begin{abstract}
We consider the time-dependent Hartree--Fock equation in $\R^3$ with the physical Coulomb interaction in both the direct and the exchange channel, the latter weighted by an exchange coupling $\eta$, linearized around spatially homogeneous equilibria $\gamma_f=g(-i\nabla)$ with compactly supported momentum profile. For the exchange-free Coulomb--Hartree equation, Nguyen and You proved that below a survival threshold $\kappa_0>0$ the linearized density carries two purely imaginary plasmon poles $\pm i\tau_0(|k|)$, uniformly separated from the particle--hole continuum on every compact subthreshold band; the available nonlinear Hartree--Fock theory, in turn, requires a short-range direct interaction and a smooth small exchange kernel and therefore excludes both Coulomb channels. On any compact band $0<k_-\le|k|\le k_+<\kappa_0$ we prove, unconditionally for the Coulomb exchange kernel: the two plasmon poles persist for small $|\eta|$, remain purely imaginary and simple, and depend real-analytically on $\eta$, with a perturbation series of uniform radius whose first-order coefficient splits into an explicit static self-energy correction and a dynamic vertex correction; each pole is realized by an exact undamped mode with a closed-form eigenfunction of unit density, so the exact collective modes are undamped on the band; the reduced fiber resolvent has rank-one poles and the causal density Green function decomposes into an explicit undamped sine wave plus a remainder whose Laplace transform is holomorphic near the poles; and the plasmon channel obeys Klein--Gordon-type dispersive decay $t^{-3/2}$ together with endpoint Strichartz estimates, uniformly in $\eta$. Exploiting the identity expressing the exchange self-energy increment as the Coulomb potential of $a_k$, we further prove a first-order Ward-type cancellation: the first-order coefficient obeys $|\tau_X(r)|\le Cr^2$ down to $r=0$, so the plasma gap is exchange-invariant at first order in the coupling, and we compute the exchange correction to the plasmon stiffness in closed form. Numerical evaluation for a $C^9$ smoothed Fermi ball satisfying the standing assumptions confirms the analysis and shows that the self-energy and vertex contributions cancel to within $1.2\%$ on the computed range, and that the exchange softens the plasmon dispersion, most strongly around $0.7\,\kappa_0$.
\end{abstract}

\maketitle

\setcounter{tocdepth}{1}
\tableofcontents

\section{Introduction}\label{sec:intro}

\subsection{The equation}
We study the time-dependent Hartree--Fock (HF) equation on $\R^3$,
\begin{equation}\label{eq:tdhf}
i\partial_t\Gamma=\bigl[-\Delta+v_C*(\rho_\Gamma-\rho_b)-\eta\,\sX_{v_C,\Gamma},\,\Gamma\bigr],
\qquad v_C(x)=\frac{1}{|x|},
\end{equation}
for a self-adjoint one-particle density operator $\Gamma(t)$ on $L^2(\R^3)$ with integral kernel $\Gamma(x,y)$, density $\rho_\Gamma(x)=\Gamma(x,x)$, a rigid neutralizing background of constant density $\rho_b$ (jellium), and exchange operator
\begin{equation}\label{eq:exchange}
\sX_{v_C,\Gamma}(x,y)=v_C(x-y)\,\Gamma(x,y).
\end{equation}
The parameter $\eta\in\R$ is an \emph{exchange coupling}: $\eta=0$ is the (Coulomb--)Hartree equation, while $\eta=1$ corresponds to the physical Hartree--Fock strength. Throughout we work in the small-exchange regime $|\eta|\ll1$ and we keep the \emph{physical Coulomb kernel in both channels}: no smoothing, screening, or cutoff is imposed on \eqref{eq:exchange}.

Equation \eqref{eq:tdhf} admits a large family of translation-invariant steady states, the Fourier multipliers
\begin{equation}\label{eq:equilibrium}
\gamma_f=g(-i\nabla),\qquad g(p)=f(|p|^2),\qquad 0\le g\le1,
\end{equation}
which include Fermi gases at zero and positive temperature; the constraint $0\le\gamma_f\le1$ is the fermionic admissibility condition. The subject of this paper is the linearization of \eqref{eq:tdhf} around \eqref{eq:equilibrium} and, specifically, the fate of the \emph{plasmon oscillations} of the Coulomb--Hartree dynamics when the Coulomb exchange term is switched on.

\subsection{Known results and the gap addressed here}\label{sec:known}
The mathematical theory of \eqref{eq:tdhf} near translation-invariant equilibria was initiated by Lewin and Sabin \cite{LewinSabin1,LewinSabin2} for the Hartree equation ($\eta=0$) with short-range interactions, where they established well-posedness in spaces of infinite-trace perturbations as well as asymptotic stability in dimension two; the theory was subsequently developed by many authors, always under the assumption that the interaction potential $w$ is short range, $\widehat w\in L^\infty$. For the \emph{Coulomb} interaction this assumption fails in an essential way. Nguyen and You \cite{NY1} showed that for compactly supported equilibria the linearized Coulomb--Hartree equation exhibits a \emph{survival threshold} $\kappa_0>0$: for $0<|k|<\kappa_0$ the spacetime symbol (Lindhard dielectric function) $D_0(\lambda,k)$ possesses exactly two simple, purely imaginary zeros
\[
\lambda_{\pm,0}(k)=\pm i\tau_0(|k|),
\]
lying strictly above the particle--hole continuum --- undamped \emph{plasmons} obeying a Klein--Gordon-type dispersion relation --- while for $|k|>\kappa_0$ the plasmons dissolve into the continuum and Landau damping sets in; see also \cite{NguyenJFA} for the closely related survival-threshold phenomenon in the kinetic setting, and \cite{MouhotVillani,BMM,HKNR} for Landau damping in the classical Vlasov theory.

On the Hartree--Fock side, the first stability result near homogeneous equilibria in the presence of a (small) exchange term is due to Collot, Danesi, de~Suzzoni and Mal\'ez\'e \cite{CDSM}, via the random-field formulation in dimensions $d\ge4$. Very recently, Nguyen and You \cite{NY2} established nonlinear Landau damping, asymptotic stability, and scattering for \eqref{eq:tdhf}-type equations in dimension $d\ge 3$ near positive equilibria, under two structural hypotheses: the direct interaction $w_1$ is short range ($\widehat w_1$ bounded --- hypothesis (H3) of \cite{NY2}), and the exchange kernel $w_2$ is smooth and small, $\widehat w_2\in W^{2n_0,\infty}$ with $\|\widehat w_2\|_{W^{2n_0,\infty}}\ll1$ (hypothesis (H4)). These hypotheses exclude the Coulomb kernel $\widehat v_C(q)=4\pi|q|^{-2}$ in \emph{both} channels, and deliberately so: under (H3) the strong Penrose--Lindhard condition holds (\cite[Theorem~3.3]{NY2}) and \emph{no plasmon poles exist at all}. Moreover, in the linear analysis of \cite{NY2} the dynamic exchange response $[\sX_{w_2,\gamma},\gamma_f]$ is not incorporated into the dielectric function --- the exchange enters the linear theory only through the equilibrium self-energy correction of the dispersion relation --- and is instead controlled perturbatively in the nonlinear iteration.

Thus the coexistence of plasmons (forced by the long-range direct Coulomb interaction) with a genuinely singular Coulomb exchange term lies outside the scope of \cite{NY1,NY2,CDSM}. The present paper provides the corresponding \emph{linear} theory on compact subthreshold wave-number bands: the persistence of the plasmon poles under the physical Coulomb exchange, their analytic dependence on the exchange coupling, an exact modal realization, the pole decomposition of the density Green function, dispersive and Strichartz estimates for the plasmon channel, and a first-order Ward-type cancellation at small wave numbers. All results are uniform on bands $I=[k_-,k_+]\Subset(0,\kappa_0)$ and unconditional for the Coulomb exchange kernel.

\subsection{Main results}\label{sec:mainresults-intro}
Fix a compact band $I=[k_-,k_+]\Subset(0,\kappa_0)$ and let $\Omega_I=\{k\in\R^3:k_-\le|k|\le k_+\}$. Write $\lambda$ for the Fourier--Laplace frequency dual to $t$ and let $D_\eta(\lambda,k)$ denote the \emph{exchange-resummed dielectric function} constructed in Section~\ref{sec:schur}: the exchange block of the linearized equation is inverted by a Neumann series and the Coulomb--Hartree density channel is then eliminated by a Feshbach--Schur complement, so that $D_\eta$ is an exact scalar function whose zeros near $\pm i\tau_0(|k|)$ are precisely the poles of the reduced fiber resolvent (Proposition~\ref{prop:schur}). Our main results (Theorems~\ref{thm:main} and~\ref{thm:ward} below) can be summarized as follows.

\begin{itemize}
\item[(a)] \emph{Persistence and pure imaginarity} (Theorem~\ref{thm:main}\,(i)). There are $\eta_0,r_0>0$ such that for $|\eta|<\eta_0$ and $k\in\Omega_I$ the function $D_\eta(\cdot,k)$ has exactly one simple zero $\lambda_{\sigma,\eta}(k)=\sigma i\tau_\eta(|k|)$ in the disk of radius $r_0$ about $\sigma i\tau_0(|k|)$, $\sigma=\pm$, with $\tau_\eta$ real, radial, and of class $C^m$.
\item[(b)] \emph{Analyticity in the exchange coupling} (Theorem~\ref{thm:main}\,(ii)). The map $\eta\mapsto\tau_\eta(r)$ is real-analytic with a convergent expansion $\tau_\eta=\tau_0+\sum_{n\ge1}\eta^n\tau_X^{(n)}$ of uniform radius on $I$; the coefficients are explicit finite-dimensional absolutely convergent integrals. The first-order coefficient splits as
\[
\tau_X=\tau_X^{\mathrm{self}}+\tau_X^{\mathrm{vertex}},
\]
a static self-energy correction plus a dynamic vertex correction, both explicit; the second-order coefficient is computed as well.
\item[(c)] \emph{Exact plasmon modes} (Theorem~\ref{thm:main}\,(iii)). Each pole is realized by an exact solution $e^{\sigma i\tau_\eta(|k|)t}\varphi_{\sigma,\eta,k}$ of the linearized equation, with the closed-form eigenfunction $\varphi_{\sigma,\eta,k}=-iV(k)E_\eta m_\eta$ of compact momentum support and unit density amplitude, $\ellf(\varphi_{\sigma,\eta,k})=1$. In particular $\pm i\tau_\eta(|k|)$ are eigenvalues, with one-dimensional eigenspaces, of the closed fiber generator on a weighted $L^2$ space, embedded in its essential spectrum $i\R$; the exact collective modes are undamped.
\item[(d)] \emph{Green function decomposition} (Theorem~\ref{thm:main}\,(iv)). The reduced fiber resolvent has rank-one residues, and the causal density Green function decomposes as $\delta(t)$ plus an explicit \emph{undamped sine wave} plus a remainder whose Laplace transform is holomorphic and uniformly bounded near the two poles.
\item[(e)] \emph{Dispersive and Strichartz estimates} (Theorem~\ref{thm:main}\,(v)). The plasmon channel obeys $\|K_{\sigma,\eta}(t)*u\|_{L^p_x}\le Ct^{-3(1/2-1/p)}\|u\|_{L^{p'}_x}$ for $t\ge1$, $2\le p\le\infty$, and the associated propagator satisfies the full range of (endpoint) Strichartz estimates, uniformly in $|\eta|<\eta_0$.
\item[(f)] \emph{First-order Ward-type cancellation} (Theorem~\ref{thm:ward}). For every $k_+<\kappa_0$ there is $C$ with
\[
|\tau_X(r)|\le C r^2,\qquad 0<r\le k_+,
\]
and $\tau_X(r)/r^2\to\beta_X$ as $r\to0$ with the closed-form stiffness correction
\[
\beta_X=\frac{1}{4\pi^4\,\omega_p^3}\iint_{\R^3\times\R^3}
\frac{\partial_1g(p)\,\partial_1g(q)\,(p_1-q_1)^2}{|p-q|^2}\dd p\dd q,
\qquad \omega_p=\tau_0(0^+).
\]
In particular the plasma gap is exchange-invariant at first order in the coupling --- a statement about the coefficient $\tau_X$ at fixed wave number, see Remark~\ref{rem:wardmeaning} --- although $\tau_X^{\mathrm{self}}$ and $\tau_X^{\mathrm{vertex}}$ are individually of order one: the cancellation is structural, being a consequence of the identity $A_{X,k}=Wa_k$ proved in Lemma~\ref{lem:AXident}.
\end{itemize}

Section~\ref{sec:numerics} complements the analysis with a numerical evaluation for a $C^9$ smoothed Fermi ball satisfying \textup{(G1)--(G2)}: the two first-order components cancel to within $1.2\%$ on the computed range $[0.02\,\kappa_0,\,0.95\,\kappa_0]$; the net shift is quadratic at small wave numbers with $\beta_X\approx-0.109$ (a softening of the plasmon stiffness by about $6.6\%$ per unit $\eta$ for that model), attains its most negative value near $0.71\,\kappa_0$, and returns towards zero as the threshold is approached. The numerics also confirm, to four digits, the closed-form expressions for $\beta_X$ and for the small-$k$ limit of the normalization integral.

\subsection{What is not proved here}\label{sec:notproved}
We state explicitly what this paper does \emph{not} prove, since the division reflects the assumptions of \cite{NY1,NY2} rather than technical convenience: the uniform limit $k\to0$ at the level of the resolvent (the Ward-type cancellation is established here for the first-order coefficient, Theorem~\ref{thm:ward}, but not for the full resolvent, which requires tracking the cancellation between the Coulomb factor $|k|^{-2}$ and the density response through the dynamic exchange vertex); the threshold regime $|k|\to\kappa_0$, where the isolated pole meets the continuum and a Plemelj boundary-value and resonance analysis, together with the first-order exchange correction of the threshold $\kappa_\eta$ and of the Landau damping rate, is required; full spectral stability on the closed right half-plane for $\eta\ne0$ (for $\eta=0$ this is part of \cite[Theorem~1.1]{NY1}); optimal pointwise-in-time decay of the continuous component of the Green function with the physical Coulomb exchange, for which the smooth-kernel estimates of \cite{NY2} are not applicable and new fractional-regularity boundary-value estimates are needed (a precise conditional criterion is given in Section~\ref{sec:conditional}); and any nonlinear statement. These are discussed in Section~\ref{sec:conclusion}.

\subsection{Method}
The proofs are elementary in the best sense: they combine (1) the observation that on bounded relative-momentum regions the Coulomb exchange operator is a \emph{weakly singular, compact} operator with an explicit Schur bound (Section~\ref{sec:exchangeop}); (2) a uniform spectral gap between the plasmon frequencies and the particle--hole continuum on compact subthreshold bands, inherited from \cite{NY1} (Lemma~\ref{lem:gap}); (3) a Neumann resummation of the dynamic exchange block followed by a Feshbach--Schur elimination of the Hartree density channel, producing the scalar dielectric function $D_\eta$ (Section~\ref{sec:schur}); (4) Rouch\'e's theorem, the real-analytic implicit function theorem and a parity symmetry, giving persistence, pure imaginarity and the $\pm$ symmetry (Section~\ref{sec:persistence}); (5) the affine dependence $L_\eta=L_0+\eta C_k$ of the fiber symbol on $\eta$, which turns the perturbation series into a geometric series with uniform radius and yields analyticity by a contour-integral argument (Section~\ref{sec:analyticity}); (6) stationary phase with a uniformly non-degenerate Hessian for the dispersive estimates (Section~\ref{sec:dispersive}); and (7) the identity $A_{X,k}=Wa_k$, which converts the difference of the self-energy and vertex contributions into a single well-conditioned integral carrying an explicit factor $(2r(p_1-q_1))^2$, whence the Ward-type cancellation (Section~\ref{sec:ward}).

\subsection*{Notation}
$\widehat u(k)=\int_{\R^3}e^{-ix\cdot k}u(x)\dd x$ and $\ellf(h)=(2\pi)^{-3}\int_{\R^3}h(p)\dd p$. We set $V(k)=\widehat{v_C}(k)=4\pi|k|^{-2}$. For $A\lesssim B$ we mean $A\le CB$ with $C$ depending only on $g$ and the band $I$. The letter $\sigma\in\{+,-\}$ indexes the two poles. Appendix~\ref{app:conventions} records the translation between our symmetric fiber coordinates and the conventions of \cite{NY1,NY2}.

\section{The model, the fiber reduction, and the Hartree input}\label{sec:model}

\subsection{Standing assumptions on the equilibrium}\label{sec:standing}
Throughout, $m\ge6$ is a fixed integer and the momentum profile $g$ in \eqref{eq:equilibrium} satisfies:
\begin{itemize}
\item[(G1)] $g(p)=f(|p|^2)$ with $f\in C^{m+3}([0,\infty))$, $0\le f\le1$, and $\supp f\subset[0,\Upsilon^2]$ for some $\Upsilon>0$; in particular $g\in C^{m+3}_c(\R^3)$ is radial with $\supp g\subset\overline{B_\Upsilon}$;
\item[(G2)] $g$ satisfies the hypotheses of \cite[Section~1.4]{NY1} for compactly supported equilibria: $g>0$ on $\{|p|<\Upsilon\}$ (connected support), and $g(p)/(\Upsilon-|p|)^{n_1}$ has a positive limit as $|p|\to\Upsilon$ for some $n_1\ge1$.
\end{itemize}
The background density in \eqref{eq:tdhf} is fixed once and for all as $\rho_b:=\rho_{\gamma_f}=(2\pi)^{-3}\int g$, so that the direct term $v_C*(\rho_\Gamma-\rho_b)$ is well defined for perturbations with decaying density and $\gamma_f$ is an exact steady state of \eqref{eq:tdhf}; equivalently, the linearized equation \eqref{eq:fibereq} below may be taken as the definition of the model. We emphasize that the regularity in (G1) is imposed on the energy-variable profile $f$, exactly as in \cite[Section~1.4]{NY1}: with $m\ge6$ this gives $f\in C^{9}$, matching the requirement $n_0\ge\frac{d+15}{2}=9$ of \cite{NY1} for $d=3$, and $g=f(|\cdot|^2)\in C^{9}(\R^3)$ follows by composition (the converse implication would not hold in finite regularity, which is why the hypothesis is placed on $f$).

The \emph{equilibrium exchange self-energy} and the exchange-corrected dispersion are
\begin{equation}\label{eq:selfenergy}
\Sigma(p)=\frac{1}{(2\pi)^3}\int_{\R^3}V(p-q)\,g(q)\dd q,\qquad
\epsilon_\eta(p)=|p|^2-\eta\,\Sigma(p).
\end{equation}
Since $g$ is a Fourier multiplier and $\sX_{v_C,\gamma_f}$ is the Fourier multiplier with symbol $\Sigma$, the equilibrium $\gamma_f$ is an exact steady state of \eqref{eq:tdhf} for every $\eta$.

\subsection{Fiberization}\label{sec:fiber}
Let $Q=\Gamma-\gamma_f$ and linearize \eqref{eq:tdhf} in $Q$. Writing $\widehat Q(p_1,p_2)$ for the momentum kernel and passing to the symmetric coordinates (center $p$, spatial wave number $k$)
\begin{equation}\label{eq:fibercoord}
q_k(t,p)=\widehat Q\Bigl(t,\,p+\tfrac k2,\,p-\tfrac k2\Bigr),\qquad
\rho_k(t)=\ellf(q_k(t)),
\end{equation}
we introduce
\begin{equation}\label{eq:aA}
a_k(p)=g\Bigl(p-\tfrac k2\Bigr)-g\Bigl(p+\tfrac k2\Bigr),\qquad
A_{\eta,k}(p)=\epsilon_\eta\Bigl(p+\tfrac k2\Bigr)-\epsilon_\eta\Bigl(p-\tfrac k2\Bigr),
\end{equation}
\begin{equation}\label{eq:W}
(Wh)(p)=\frac{1}{(2\pi)^3}\int_{\R^3}V(p-q)\,h(q)\dd q
=\frac{1}{2\pi^2}\int_{\R^3}\frac{h(q)}{|p-q|^2}\dd q .
\end{equation}
Note $A_{0,k}(p)=2p\cdot k$ and $S_k:=\supp a_k\subset\{|p|\le\Upsilon+\tfrac{|k|}2\}$.

\begin{lemma}[Linearized fiber equation]\label{lem:fiber}
The linearization of \eqref{eq:tdhf} around $\gamma_f$ decouples over the spatial wave number $k\ne0$:
\begin{equation}\label{eq:fibereq}
i\partial_t q_k=A_{\eta,k}\,q_k+a_k\,V(k)\,\rho_k-\eta\,a_k\,Wq_k .
\end{equation}
Distinct wave numbers do not couple, while the exchange term $W$ couples the relative momenta within each fiber nonlocally.
\end{lemma}

\begin{proof}
With $h_{\eta,f}=-\Delta-\eta\,\sX_{v_C,\gamma_f}$ the linearized equation reads
$i\partial_tQ=[h_{\eta,f},Q]+[v_C*\rho_Q-\eta\,\sX_{v_C,Q},\,\gamma_f]$.
The first commutator has $(p+\frac k2,p-\frac k2)$ entry $A_{\eta,k}(p)q_k(p)$. Since $\widehat{\rho_Q}(k)=(2\pi)^{-3}\int\widehat Q(q+\frac k2,q-\frac k2)\dd q=\rho_k$, one computes
$[v_C*\rho_Q,\gamma_f](p+\frac k2,p-\frac k2)=a_k(p)V(k)\rho_k$. Finally the two-variable Fourier transform of $v_C(x-y)Q(x,y)$ gives
$\widehat{\sX_{v_C,Q}}(p+\frac k2,p-\frac k2)=(2\pi)^{-3}\int V(p-q)q_k(q)\dd q=(Wq_k)(p)$,
so that $[-\eta\sX_{v_C,Q},\gamma_f]$ contributes $-\eta a_k(p)(Wq_k)(p)$.
\end{proof}

Throughout, let
\begin{equation}\label{eq:Xw}
X_w:=\bigl\{q:\ \langle p\rangle^{2}q\in L^2(\R^3)\bigr\},\qquad
\|q\|_{X_w}:=\|\langle p\rangle^{2}q\|_{L^2},\qquad \langle p\rangle=(1+|p|^2)^{1/2}.
\end{equation}
Since $\langle p\rangle^{-2}\in L^2(\R^3)$, one has $X_w\hookrightarrow L^1\cap L^2(\R^3)$ and the density functional $\ellf$ is \emph{bounded} on $X_w$. The weight is the reason the fiber flow admits a closed generator: on $L^2$ itself the functional $\ellf$ is unbounded and not even closable (take $q_n(p)=n^{-2}\varphi(p_1)\psi(p_\perp/n)$: then $q_n\to0$ in $L^2$ while $\ellf(q_n)$ is constant), so no closed realization of the generator exists there.

\begin{lemma}[Well-posedness of the fiber flow]\label{lem:wp}
For every $k\ne0$, $\eta\in\R$, and $q_{k,0}\in X_w$, the Duhamel formulation of \eqref{eq:fibereq},
\[
q_k(t)=e^{-iA_{\eta,k}t}q_{k,0}+\int_0^te^{-iA_{\eta,k}(t-s)}\Bigl[-iV(k)a_k\,\ellf(q_k(s))+i\eta\,a_k(Wq_k)(s)\Bigr]\dd s,
\]
has a unique global mild solution $q_k\in C(\R;X_w)$, and $\|q_k(t)\|_{X_w}\le e^{C_k|t|}\|q_{k,0}\|_{X_w}$. Consequently the Fourier--Laplace transform $\widetilde q_k(\lambda)=\int_0^\infty e^{-\lambda t}q_k(t)\dd t$ is well defined and analytic in $\{\Re\lambda>C_k\}$, with values in $X_w$.
\end{lemma}

\begin{proof}
The multiplier group $e^{-iA_{\eta,k}t}$ is isometric on $X_w$ (the weight commutes with multiplication operators and $A_{\eta,k}$ is real). The map $q\mapsto a_kV(k)\ellf(q)$ is a bounded rank-one operator on $X_w$, since $\ellf$ is bounded on $X_w$ and $a_k\in C_c^{m+3}$. By the Hardy--Littlewood--Sobolev inequality \cite{Stein}, $W:L^2(\R^3)\to L^6(\R^3)$ is bounded (the kernel of \eqref{eq:W} is a constant multiple of the Riesz kernel $|x|^{-2}$), and hence $q\mapsto a_kWq$ maps $X_w\subset L^2$ boundedly into $L^2(S_k)\subset X_w$ by H\"older on the compact set $S_k$. The claim follows from the contraction mapping principle and iteration.
\end{proof}

\subsection{The Hartree plasmon input}\label{sec:input}
For $\eta=0$ the density obeys a scalar resolvent equation with the Penrose--Lindhard function
\begin{equation}\label{eq:D0}
D_0(\lambda,k)=1+iV(k)\,\ellf\Bigl(\frac{a_k}{\lambda+iA_{0,k}}\Bigr),
\end{equation}
which coincides, after the change of coordinates recorded in Appendix~\ref{app:conventions}, with the dielectric function of \cite{NY1}.

\begin{assumption}[Hartree plasmon input; \cite{NY1}]\label{ass:input}
There exist a survival threshold $\kappa_0>0$ and a radial function $\tau_0$, of class $C^{m+1}$ on compact subsets of $(0,\kappa_0)$, such that
\begin{equation}\label{eq:inputzeros}
D_0(\pm i\tau_0(|k|),k)=0,\qquad 0<|k|<\kappa_0,
\end{equation}
these are the only zeros of $D_0(\cdot,k)$ on $i\R$, they are simple, and
\begin{equation}\label{eq:inputbounds}
\tau_0(r)>2r\Upsilon+r^2,\qquad
c_0 r\le \tau_0'(r)\le C_0 r,\qquad
c_0\le\tau_0''(r)\le C_0,\qquad 0<r<\kappa_0,
\end{equation}
for some $c_0,C_0>0$; moreover $\tau_0(r)\to\omega_p:=\sqrt{8\pi\rho_g}>0$ as $r\to0^+$, where $\rho_g=(2\pi)^{-3}\int g$.
\end{assumption}

\begin{remark}\label{rem:inputjustified}
Assumption~\ref{ass:input} is not a new hypothesis: it reduces \emph{exactly} to \cite{NY1}. Setting $\mu:=g/(2\pi^2)$, the change of variables of Appendix~\ref{app:conventions} identifies \eqref{eq:D0} with the dielectric function of \cite{NY1} for the kernel $\widehat w(k)=|k|^{-2}$ and the equilibrium profile $\mu$; note $0\le\mu\le(2\pi^2)^{-1}\le1$, and under \textup{(G1)--(G2)} the profile $\mu$ satisfies the hypotheses of \cite[Section~1.4]{NY1}, the regularity being imposed there, as in \textup{(G1)}, on the energy-variable profile. Assumption~\ref{ass:input} is then the content of \cite[Theorems~1.1 and~2.6]{NY1} together with their proofs. Specifically, the existence, uniqueness on the axis, and the two-sided bounds are \cite[Theorem~2.6 and (2.24)]{NY1}; the strict separation $\tau_0>2r\Upsilon+r^2$ for $r<\kappa_0$ and the simplicity $\partial_\tau D_0(i\tau,k)>0$ at the zero are established inside the proof of \cite[Theorem~2.6]{NY1}; and no zeros exist in $\{\Re\lambda>0\}$ by \cite[Theorem~1.1]{NY1}. The numerical identities of \cite{NY1} transfer through $\rho_\mu=\int\mu=(2\pi^2)^{-1}(2\pi)^3\rho_g=4\pi\rho_g$: for instance $\tau_*(0)=\sqrt{2\rho_\mu}$ of \cite[(1.17)]{NY1} becomes $\tau_0(0^+)=\omega_p=\sqrt{8\pi\rho_g}$, consistent with the exact identity $\int_{\R}\varphi_g(u)\dd u=(2\pi)^3\rho_g=\pi^2\omega_p^2$ used in Section~\ref{sec:ward}.
\end{remark}

Throughout the paper we fix a compact band
\begin{equation}\label{eq:band}
I=[k_-,k_+]\Subset(0,\kappa_0),\qquad
\Omega_I=\{k\in\R^3:\ k_-\le|k|\le k_+\},
\end{equation}
and we set $\delta_I:=\min_{r\in I}\bigl\{\tau_0(r)-(2r\Upsilon+r^2)\bigr\}>0$.

\section{Main results}\label{sec:main}

Let $E_\eta$, $m_\eta$, and the exchange-resummed dielectric function $D_\eta(\lambda,k)$ be as constructed in Section~\ref{sec:schur}, and let $F_\eta(\tau,k)=D_\eta(i\tau,k)$ denote its (real-valued) restriction to the imaginary axis. Let $L_{0,*}=\tau_0(r)+A_{0,k}$ and $C_k=A_{X,k}-a_kW_R$ with $A_{X,k}(p)=\Sigma(p-\frac k2)-\Sigma(p+\frac k2)$, and write $L_{0,*}^{-1}a_k$ for the zero extension of $a_k/L_{0,*}$ off $S_k$ (Section~\ref{sec:analyticity}).

\begin{theorem}[Plasmons under Coulomb exchange on subthreshold bands]\label{thm:main}
Let $g$ satisfy \textup{(G1)--(G2)}, let Assumption~\textup{\ref{ass:input}} hold, and fix $I\Subset(0,\kappa_0)$ as in \eqref{eq:band}. Then there exist $\eta_0>0$, $r_0>0$, and $\eta_1\in(0,\eta_0]$ such that for all $|\eta|<\eta_0$, $k\in\Omega_I$, and $\sigma\in\{+,-\}$ the following hold.
\begin{itemize}
\item[(i)] \textup{(Persistence, uniqueness, pure imaginarity.)} $D_\eta(\cdot,k)$ has exactly one zero
\[
\lambda_{\sigma,\eta}(k)=\sigma\, i\,\tau_\eta(|k|)
\]
in the disk $U_{\sigma,k}=\{\lambda:|\lambda-\sigma i\tau_0(|k|)|<r_0\}$; it is simple, purely imaginary, and $\tau_\eta$ is real, radial, and of class $C^m$ on $I$.
\item[(ii)] \textup{(Real-analyticity in $\eta$; explicit expansion.)} For $|\eta|<\eta_1$,
\[
\tau_\eta(r)=\tau_0(r)+\sum_{n\ge1}\eta^n\,\tau_X^{(n)}(r),
\]
absolutely and uniformly on $I$, with $\tau_X^{(n)}\in C^{m-2}(I)$; the residues $r_{\sigma,\eta}(k)$ of $D_\eta^{-1}$ are real-analytic in $\eta$ as well. The first coefficient is $\tau_X=\tau_X^{\mathrm{self}}+\tau_X^{\mathrm{vertex}}$ with
\begin{equation}\label{eq:tauXsplit}
\tau_X^{\mathrm{self}}(r)=-\frac{\ellf\bigl(L_{0,*}^{-1}A_{X,k}L_{0,*}^{-1}a_k\bigr)}{\ellf\bigl(L_{0,*}^{-2}a_k\bigr)},
\qquad
\tau_X^{\mathrm{vertex}}(r)=+\frac{\ellf\bigl(L_{0,*}^{-1}a_kW_RL_{0,*}^{-1}a_k\bigr)}{\ellf\bigl(L_{0,*}^{-2}a_k\bigr)},
\end{equation}
and the second coefficient $\tau_X^{(2)}$ is given by \eqref{eq:tau2} below. Moreover
$\|\tau_\eta-\tau_0-\eta\tau_X\|_{C^0(I)}\le C_I|\eta|^2$ for $|\eta|\le\eta_1/2$.
\item[(iii)] \textup{(Exact plasmon modes; undamped collective oscillations.)} The function
\[
\varphi_{\sigma,\eta,k}:=-iV(k)\,E_\eta(\lambda_{\sigma,\eta},k)\,m_\eta(\lambda_{\sigma,\eta},k)\ \in\ H^1(\R^3),
\qquad \supp\varphi_{\sigma,\eta,k}\subset S_k,
\qquad \ellf(\varphi_{\sigma,\eta,k})=1,
\]
yields the exact solution $q_k(t)=e^{\sigma i\tau_\eta(|k|)t}\varphi_{\sigma,\eta,k}$ of \eqref{eq:fibereq}; conversely every exponential solution with frequency in $U_{\sigma,k}$ is a multiple of it. Thus $\pm i\tau_\eta(|k|)$ are eigenvalues of the closed operator $\sG_{\eta,k}$ of Section~\ref{sec:modes}, with one-dimensional eigenspaces, embedded in its essential spectrum $\sigma_{\mathrm{ess}}(\sG_{\eta,k})=i\R$; in particular the exact collective modes are undamped.
\item[(iv)] \textup{(Rank-one residues; sine-form Green function.)} Let $\eta$ be real. The reduced fiber resolvent $\sA_\eta^{-1}$ of \eqref{eq:Ainv} is meromorphic on $U_{\sigma,k}$ with the rank-one residue \eqref{eq:res37}. The function $D_\eta(\cdot,k)^{-1}-1$ belongs to the Hardy space $H^2$ of every half-plane $\{\Re\lambda>c\}$, $c>\Lambda_I$ \textup{(}Lemma~\ref{lem:halfplane}\textup{)}, so by the Paley--Wiener theorem it is the Laplace transform of a unique causal function; the resulting causal density Green function $\widehat G_\eta$ of Proposition~\ref{prop:green} decomposes as
\[
\widehat G_\eta(t,k)=\delta(t)-\frac{2}{\partial_\tau F_\eta(\tau_\eta(|k|),k)}\,\sin\bigl(\tau_\eta(|k|)t\bigr)\mathbf 1_{t\ge0}+\widehat G^{\mathrm{rem}}_\eta(t,k),
\]
where $\mathcal L[\widehat G^{\mathrm{rem}}_\eta]$ is the restriction of the \emph{unique meromorphic continuation} of $D_\eta^{-1}-1$ to the connected set $\Omega_k=\{\Re\lambda>0\}\cup U_{+,k}\cup U_{-,k}$ \textup{(}Lemma~\ref{lem:continuation}\textup{)} minus its two principal parts: it is holomorphic on $\{\Re\lambda>\Lambda_I\}\cup U_{+,k}\cup U_{-,k}$ and uniformly bounded on the two disks, and its only possible further singularities lie in the fixed bounded set $Q_k\subset\{0<\Re\lambda\le\Lambda_I\}$ of Lemma~\ref{lem:continuation}.
\item[(v)] \textup{(Dispersive and Strichartz estimates.)} For $\chi\in C_c^{m-1}(\Omega_I)$ let
\[
K_{\sigma,\eta}(t,x)=\frac{1}{(2\pi)^{3}}\int_{\R^3} e^{ik\cdot x+\sigma it\tau_\eta(|k|)}\,\chi(k)\,r_{\sigma,\eta}(k)\dd k .
\]
Then for $t\ge1$ and $2\le p\le\infty$,
\[
\|K_{\sigma,\eta}(t)*u\|_{L^p_x(\R^3)}\le C\,t^{-3(\frac12-\frac1p)}\|u\|_{L^{p'}_x(\R^3)},
\]
and for every $\widetilde\chi\in C_c^{m-1}(\Omega_I)$ and every admissible pair $\frac1q+\frac3{2p}=\frac34$, $q\ge2$ \textup{(}endpoint included\textup{)},
\[
\bigl\|e^{\sigma it\tau_\eta(|D|)}\widetilde\chi(D)f\bigr\|_{L^q_t(\R;L^p_x(\R^3))}\le C\|f\|_{L^2(\R^3)} .
\]
All constants are uniform in $|\eta|<\eta_0$ and $\sigma$.
\end{itemize}
\end{theorem}

\begin{theorem}[First-order Ward-type cancellation of the plasma gap]\label{thm:ward}
Let $g$ satisfy \textup{(G1)--(G2)} and Assumption~\textup{\ref{ass:input}}, and fix $k_+<\kappa_0$. Then the first-order coefficient $\tau_X$ of Theorem~\textup{\ref{thm:main}\,(ii)} extends to $(0,k_+]$ and satisfies
\[
|\tau_X(r)|\ \le\ C(k_+)\,r^2,\qquad 0<r\le k_+ .
\]
Moreover the limit $\beta_X=\lim_{r\to0^+}\tau_X(r)/r^2$ exists and equals
\begin{equation}\label{eq:betaX}
\begin{split}
\beta_X&=\frac{1}{4\pi^4\,\omega_p^3}\iint_{\R^3\times\R^3}\frac{\partial_1g(p)\,\partial_1g(q)\,(p_1-q_1)^2}{|p-q|^2}\dd p\dd q\\
&=-\frac{1}{4\pi^4\,\omega_p^3}\iint_{\R^3\times\R^3}\frac{\partial_1g(p)\,\partial_1g(q)\,|p_\perp-q_\perp|^2}{|p-q|^2}\dd p\dd q,
\end{split}
\end{equation}
where $p=(p_1,p_\perp)$. In particular $\partial_\eta\tau_\eta(r)\big|_{\eta=0}=\tau_X(r)\to0$ as $r\to0^+$, at the quadratic rate $\beta_Xr^2$: the plasma gap is invariant under the Coulomb exchange at first order in the coupling. This is a statement about the first-order coefficient at fixed $r$; a joint expansion uniform as $(r,\eta)\to(0,0)$ would require resolvent control uniform down to $k=0$, which remains open \textup{(}Section~\ref{sec:conclusion}\textup{)}.
\end{theorem}

\begin{remark}[The dynamic vertex is not a correction term]\label{rem:meaning}
Theorem~\ref{thm:main} is \emph{not} obtained by inserting the exchange into the equilibrium dispersion $\epsilon_\eta$ alone. The dynamic exchange block $-\eta a_kW$ of \eqref{eq:fibereq} is inverted first, and the Coulomb--Hartree density channel is eliminated afterwards by a Schur complement; the first-order shift then necessarily carries the two terms of \eqref{eq:tauXsplit}. Keeping only $\tau_X^{\mathrm{self}}$ --- i.e., correcting the dispersion relation but discarding the vertex --- is inconsistent as a first-order theory: by Theorem~\ref{thm:ward} the two terms cancel to leading order at small $k$ (and, numerically, to within $1.2\%$ on the computed range $[0.02\kappa_0,0.95\kappa_0]$, with $\tau_X^{\mathrm{self}}$ alone off even in sign; see Section~\ref{sec:numerics}). This also contrasts with the linear theory of \cite{NY2}, where the dynamic exchange is kept outside the dielectric function and controlled perturbatively in the nonlinear iteration --- an option available there because, the direct interaction being short range, no plasmon poles exist.
\end{remark}

\begin{remark}[Scope]\label{rem:scope}
The constants $\eta_0,\eta_1$ degenerate as $k_-\to0$ or $k_+\to\kappa_0$; the theorem is a statement about fixed compact subthreshold bands. See Sections~\ref{sec:notproved} and~\ref{sec:conclusion} for the precise list of open endpoints.
\end{remark}

\section{The Coulomb exchange operator on the band}\label{sec:exchangeop}

\subsection{Regularity of the self-energy}
\begin{lemma}\label{lem:sigma}
Under \textup{(G1)}, $\Sigma\in C^{m+3}(\R^3)$ is radial and $\|\partial^\alpha\Sigma\|_{L^\infty}\le C_{\alpha,g}$ for $|\alpha|\le m+3$. Consequently, with $A_{X,k}(p)=\Sigma(p-\frac k2)-\Sigma(p+\frac k2)$ one has $A_{\eta,k}=A_{0,k}+\eta A_{X,k}$ and
$\|A_{\eta,k}-A_{0,k}\|_{C^{m+2}}\le C|\eta|$, uniformly for $k\in\Omega_I$. Moreover $|\nabla A_{\eta,k}|\ge 2k_- - C|\eta|$ on $\R^3$.
\end{lemma}

\begin{proof}
$V(p)=4\pi|p|^{-2}$ is locally integrable and tempered, so $\partial^\alpha\Sigma=(2\pi)^{-3}V*\partial^\alpha g$. Splitting the convolution at $|p-q|=1$,
$\int_{|p-q|<1}|p-q|^{-2}\dd q=4\pi$ and $|p-q|^{-2}\le1$ otherwise, whence
$|(V*\partial^\alpha g)(p)|\le C(\|\partial^\alpha g\|_\infty+\|\partial^\alpha g\|_{L^1})$. Radiality follows from rotation invariance. The last two claims follow since $\nabla A_{0,k}=2k$.
\end{proof}

\subsection{Boundedness, compactness, and localization}
Fix once and for all
\begin{equation}\label{eq:R}
R:=\Upsilon+\tfrac{k_+}2+1,\qquad X:=L^2(B_R),
\end{equation}
so that $S_k\subset B_{R-1}$ for all $k\in\Omega_I$, and let $W_R$ denote the operator \eqref{eq:W} with the integration restricted to $B_R$.

\begin{lemma}[Schur bound and compactness]\label{lem:WR}
$W_R\in\mathcal B(X)$ with $\|W_R\|_{X\to X}\le CR$; $W_R$ is compact, maps real functions to real functions, and commutes with the parity $(Uh)(p)=h(-p)$. Moreover
\begin{equation}\label{eq:tail}
\bigl\|\mathbf 1_{S_k}\,W\,\mathbf 1_{\{|q|>R\}}\bigr\|_{L^2(\R^3)\to L^2(\R^3)}\le CR^{-1/2}.
\end{equation}
\end{lemma}

\begin{proof}
Schur's test with the constant weight: $\sup_{p\in B_R}\int_{B_R}4\pi|p-q|^{-2}\dd q\le C\int_0^{2R}r^{-2}r^2\dd r\le CR$, and symmetrically in $q$. For compactness, split the kernel at $|p-q|=\delta$: the far part is square integrable on $B_R\times B_R$ (hence Hilbert--Schmidt), while the near part has Schur norm $O(\delta)$; thus $W_R$ is a norm limit of compact operators. (Note that $|p-q|^{-4}$ is \emph{not} integrable across the diagonal in $\R^3\times\R^3$, so the splitting is necessary.) Reality and parity follow from $V(p-q)=\overline{V(p-q)}=V((-p)-(-q))$. For \eqref{eq:tail}, if $p\in S_k$ and $|q|>R$ then $|p-q|\ge|q|-|p|\ge|q|-(R-1)\gtrsim\langle q\rangle$, so the tail operator has kernel dominated by $C|q|^{-2}$ and
$\|\mathbf 1_{S_k}W\mathbf 1_{|q|>R}h\|_{L^\infty}\le C\bigl(\int_{|q|>R}|q|^{-4}\dd q\bigr)^{1/2}\|h\|_{L^2}\le CR^{-1/2}\|h\|_{L^2}$, and $S_k$ has bounded measure.
\end{proof}

\begin{remark}[Reduction to the ball]\label{rem:reduction}
The pole analysis below closes on $X=L^2(B_R)$ for the following reason. Off $S_k$ the equation \eqref{eq:fibereq} is the free flow $i\partial_tq=A_{\eta,k}q$, which is solved explicitly; the exterior component feeds back into the equation on $S_k$ only through the two \emph{known source terms} $a_kV(k)\ellf(q^{\mathrm{out}}(t))$ and $\eta a_kWq^{\mathrm{out}}(t)$, the second of which is controlled by \eqref{eq:tail} and by Lemma~\ref{lem:WR} on $B_R\setminus S_k$. Neither term affects the homogeneous (pole) problem: an exponential solution supported where $a_k=0$ satisfies $(\lambda+iA_{\eta,k})q=0$ and hence vanishes (see the proof of Theorem~\ref{thm:main}\,(iii) in Section~\ref{sec:modes}). The precise class of initial data for which the density resolvent is meromorphic across the axis is discussed in Remark~\ref{rem:generaldata}.
\end{remark}

\subsection{Uniform gap}
\begin{lemma}[Uniform gap on the band]\label{lem:gap}
There exist $c_I>0$, $r_I>0$, and $\eta_I>0$ such that for all $k\in\Omega_I$, $\sigma=\pm$, $|\eta|\le\eta_I$ \textup{(}$\eta$ possibly complex\textup{)}, and $|\lambda-\sigma i\tau_0(|k|)|\le r_I$,
\begin{equation}\label{eq:gap}
\inf_{p\in S_k}\bigl|\lambda+iA_{\eta,k}(p)\bigr|\ \ge\ c_I .
\end{equation}
\end{lemma}

\begin{proof}
If $a_k(p)\ne0$ then $p\mp\frac k2\in\supp g$ for one of the signs, so $|p|\le\Upsilon+\frac{|k|}2$ and $|A_{0,k}(p)|=|2p\cdot k|\le2|k|\Upsilon+|k|^2$. By \eqref{eq:band}, $\tau_0(|k|)-|A_{0,k}(p)|\ge\delta_I$ on $S_k$. By Lemma~\ref{lem:sigma}, $\|A_{\eta,k}-A_{0,k}\|_\infty\le C|\eta|$ (this bound is valid verbatim for complex $\eta$, since $A_{X,k}$ is a fixed real function). Choosing $r_I<\delta_I/4$ and $C\eta_I<\delta_I/4$ gives \eqref{eq:gap} with $c_I=\delta_I/2$.
\end{proof}

Throughout, for $\lambda$ in the disks of Lemma~\ref{lem:gap} we write
\begin{equation}\label{eq:mK}
m_\eta(\lambda,k;p):=\frac{a_k(p)}{\lambda+iA_{\eta,k}(p)}\quad(\text{extended by }0\text{ off }S_k),
\qquad
K_\eta(\lambda,k):=m_\eta(\lambda,k)\,W_R .
\end{equation}
Since the zero set of the denominator is uniformly separated from $S_k$, the zero extension $m_\eta$ is of class $C^{m}$ in $(k,p)$, holomorphic in $\lambda$ on the disks, and holomorphic in $\eta$ for $|\eta|\le\eta_I$.

\section{Exchange resummation and the effective dielectric function}\label{sec:schur}

Taking the Fourier--Laplace transform of \eqref{eq:fibereq} for $\Re\lambda$ large (Lemma~\ref{lem:wp}) gives
\begin{equation}\label{eq:laplace}
\bigl(\lambda+iA_{\eta,k}-i\eta\,a_kW_R\bigr)\widetilde q_k+iV(k)\,a_k\,\ellf(\widetilde q_k)=q_{k,0}
\end{equation}
for data supported in $B_R$ (Remark~\ref{rem:reduction}). Division by $\lambda+iA_{\eta,k}$ is always legitimate for $\Re\lambda>0$ (there $|\lambda+iA_{\eta,k}|\ge\Re\lambda$ pointwise); on the \emph{closed} disks of Lemma~\ref{lem:gap} it is legitimate when applied to functions supported in $S_k$ --- more generally, in a compact set on which \eqref{eq:gap} holds. Writing $s_{k,0}=(\lambda+iA_{\eta,k})^{-1}q_{k,0}$,
\begin{equation}\label{eq:reduced}
\bigl(\Id-i\eta K_\eta\bigr)\widetilde q_k+iV(k)\,m_\eta\,\ellf(\widetilde q_k)=s_{k,0}.
\end{equation}

\begin{lemma}[Uniform resummation of the exchange block]\label{lem:neumann}
There is $\eta_0>0$ such that for $|\eta|<\eta_0$ \textup{(}complex\textup{)}, $k\in\Omega_I$, and $\lambda$ in the disks of Lemma~\textup{\ref{lem:gap}},
\begin{equation}\label{eq:E}
E_\eta(\lambda,k):=\bigl(\Id-i\eta K_\eta(\lambda,k)\bigr)^{-1}\in\mathcal B(X),\qquad
\sup\|E_\eta\|_{X\to X}\le C_I,
\end{equation}
and moreover $\|E_\eta-\Id\|_{C^m_{k,\lambda}(\mathcal B(X))}\le C_I|\eta|$.
\end{lemma}

\begin{proof}
$\|K_\eta\|\le\|m_\eta\|_\infty\|W_R\|\le(\|a_k\|_\infty/c_I)\,CR$ by Lemmas~\ref{lem:WR}--\ref{lem:gap}; choose $\eta_0$ with $|\eta|\|K_\eta\|\le\frac12$ and expand in a Neumann series. Derivatives follow from $\partial E_\eta=i\eta E_\eta(\partial K_\eta)E_\eta$ iterated; holomorphy in $(\lambda,\eta)$ is inherited from $m_\eta$.
\end{proof}

\begin{definition}[Exchange-resummed dielectric function]\label{def:Deta}
For $|\eta|<\eta_0$, $k\in\Omega_I$, and $\lambda$ in the disks of Lemma~\ref{lem:gap},
\begin{equation}\label{eq:Deta}
D_\eta(\lambda,k):=1+iV(k)\,\ellf\bigl(E_\eta(\lambda,k)\,m_\eta(\lambda,k)\bigr).
\end{equation}
\end{definition}
Note that $E_\eta m_\eta$ is supported in $S_k$, so $\ellf$ is applied to an $X$-function of fixed compact support; $D_\eta$ is holomorphic in $(\lambda,\eta)$ and of class $C^m$ in $k$, and $D_0$ agrees with \eqref{eq:D0}.

\begin{proposition}[Feshbach--Schur reduction]\label{prop:schur}
Set $\sA_\eta(\lambda,k):=\Id-i\eta K_\eta+iV(k)\,m_\eta\otimes\ellf$ on $X$, so that \eqref{eq:reduced} reads $\sA_\eta\widetilde q_k=s_{k,0}$. If $D_\eta(\lambda,k)\ne0$, then $\sA_\eta$ is invertible with
\begin{equation}\label{eq:Ainv}
\sA_\eta^{-1}=E_\eta-\frac{iV(k)}{D_\eta(\lambda,k)}\,E_\eta m_\eta\otimes\ellf\,E_\eta,
\end{equation}
and the density is
\begin{equation}\label{eq:rhotilde}
\widetilde\rho_k(\lambda)=\ellf(\widetilde q_k)=\frac{\ellf\bigl(E_\eta(\lambda,k)\,s_{k,0}\bigr)}{D_\eta(\lambda,k)} .
\end{equation}
Consequently: \textup{(a)} the reduced operator family $\lambda\mapsto\sA_\eta(\lambda,k)^{-1}$ is meromorphic on the disks of Lemma~\textup{\ref{lem:gap}}, with poles exactly at the zeros of $D_\eta$ --- each zero is a genuine pole of the operator-valued family, since the residue \eqref{eq:res37} is a nonzero rank-one operator ($E_\eta$ is invertible, $m_\eta\ne0$, and $\ellf\ne0$); \textup{(b)} for initial data $q_{k,0}$ supported in $S_k$ --- more generally, in a compact set on which \eqref{eq:gap} holds --- the factor $s_{k,0}$ is holomorphic on the disks, so the density resolvent \eqref{eq:rhotilde} is meromorphic there and its poles are \emph{contained in} the zero set of $D_\eta$: for particular data the scalar numerator $\ellf(E_\eta s_{k,0})$ may vanish at a zero of $D_\eta$ and cancel the pole, so no claim of exactness is made at the level of a fixed datum.
\end{proposition}

\begin{remark}[Reduced versus full resolvent]\label{rem:reducedvsfull}
The operator $\sA_\eta^{-1}$ of \eqref{eq:Ainv} is the inverse of the Feshbach--Schur (Birman--Schwinger-type) \emph{reduced} operator of the interaction region, not the resolvent of the full fiber generator: the two differ by the free factor $(\lambda+iA_{\eta,k})^{-1}$, which for general data is holomorphic only in $\{\Re\lambda>0\}$ and possesses resonant sets inside $B_R\setminus S_k$ as well as outside $B_R$. All pole statements in this paper concern the reduced family and, via \textup{(b)} above, the density response to gap-localized data; the boundary behavior for general data is discussed in Remark~\ref{rem:generaldata}.
\end{remark}

\begin{proof}
Apply $\ellf$ to \eqref{eq:reduced} after inverting $\Id-i\eta K_\eta$:
$\widetilde q_k=E_\eta s_{k,0}-iV\ellf(\widetilde q_k)E_\eta m_\eta$, hence
$\ellf(\widetilde q_k)\bigl(1+iV\ellf(E_\eta m_\eta)\bigr)=\ellf(E_\eta s_{k,0})$, which is \eqref{eq:rhotilde}. For \eqref{eq:Ainv} one multiplies out and uses
$iV-\frac{iV}{D_\eta}-\frac{(iV)^2}{D_\eta}\ellf(E_\eta m_\eta)
=\frac{iV}{D_\eta}\bigl(D_\eta-1-iV\ellf(E_\eta m_\eta)\bigr)=0$.
\end{proof}

\begin{remark}\label{rem:scalarbutexact}
Although $D_\eta$ is scalar, it contains the full operator inverse $E_\eta=(\Id-i\eta m_\eta W_R)^{-1}$: it is the exact Feshbach reduction of the operator-valued linear response, not a dielectric function in which the dynamic exchange has been discarded as an external forcing.
\end{remark}

\section{Persistence and pure imaginarity}\label{sec:persistence}

\begin{lemma}\label{lem:Dpert}
For $|\eta|<\eta_0$ \textup{(}complex\textup{)} and $\lambda$ on the closed disks $\overline{U_{\sigma,k}}$ with $r_0\le r_I$,
\begin{equation}\label{eq:Dpert}
\|D_\eta-D_0\|_{C^m_{k,\lambda}}\le C_I|\eta| ,
\end{equation}
and $D_\eta(\cdot,k)$ is holomorphic on the disks.
\end{lemma}

\begin{proof}
$D_\eta-D_0=iV(k)\bigl[\ellf((E_\eta-\Id)m_\eta)+\ellf(m_\eta-m_0)\bigr]$; the first term is $O(|\eta|)$ by Lemma~\ref{lem:neumann} and the second by Lemma~\ref{lem:sigma} and the gap; $V(k)$ is bounded on $\Omega_I$ since $k_->0$. Higher derivatives are handled by the same bounds.
\end{proof}

\begin{proof}[Proof of Theorem~\ref{thm:main}\,(i)]
\emph{Persistence and uniqueness.} By Assumption~\ref{ass:input} the zeros $\pm i\tau_0(|k|)$ of $D_0(\cdot,k)$ are simple, and the set $\mathcal Z_0=\{(k,\lambda):k\in\Omega_I,\ \lambda=\pm i\tau_0(|k|)\}$ is compact. As detailed in Appendix~\ref{app:disks}, one may choose $r_0\in(0,r_I]$, uniform on $\Omega_I$, so that the disks $U_{+,k},U_{-,k}$ are disjoint, avoid the particle--hole region, and each contains exactly one (simple) zero of $D_0(\cdot,k)$; then
\begin{equation}\label{eq:d0}
d_0:=\min\bigl\{|D_0(\lambda,k)|:\ k\in\Omega_I,\ \sigma=\pm,\ |\lambda-\sigma i\tau_0(|k|)|=r_0\bigr\}>0 .
\end{equation}
If $C_I|\eta|<d_0$, Rouch\'e's theorem and \eqref{eq:Dpert} give exactly one zero $\lambda_{\sigma,\eta}(k)$ of $D_\eta(\cdot,k)$ in $U_{\sigma,k}$, counted with multiplicity; since the count is one, the zero is simple.

\emph{Pure imaginarity.} Restrict to $\lambda=i\tau$, $\tau$ real. On the disks,
\begin{equation}\label{eq:breal}
b_\eta(\tau,k;p):=\frac{a_k(p)}{\tau+A_{\eta,k}(p)}\quad\text{is real},\qquad
E^{\R}_\eta(\tau,k):=\bigl(\Id-\eta\,b_\eta W_R\bigr)^{-1}
\end{equation}
is a real operator (Lemma~\ref{lem:WR}), and $m_\eta(i\tau,k)=-i\,b_\eta(\tau,k)$, $E_\eta(i\tau,k)=E^\R_\eta(\tau,k)$, so that
\begin{equation}\label{eq:F}
F_\eta(\tau,k):=D_\eta(i\tau,k)=1+V(k)\,\ellf\bigl(E^\R_\eta(\tau,k)\,b_\eta(\tau,k)\bigr)\in\R .
\end{equation}
By simplicity of the Hartree zero, $\partial_\tau F_0(\tau_0(|k|),k)=i\,\partial_\lambda D_0\ne0$, and by compactness it is uniformly bounded away from zero on $\Omega_I$. The real-analytic implicit function theorem (with the uniform $C^1$ bounds of Lemma~\ref{lem:Dpert}) yields, for $|\eta|$ small, a unique real root $\tau_\eta(|k|)$ of $F_\eta(\cdot,k)$ near $\tau_0(|k|)$; then $i\tau_\eta(|k|)\in U_{+,k}$ is a zero of $D_\eta(\cdot,k)$ and, by the uniqueness from Rouch\'e's step, $\lambda_{+,\eta}(k)=i\tau_\eta(|k|)$.

\emph{The negative branch and radiality.} With the parity $U$ of Lemma~\ref{lem:WR} one has $Ua_kU=-a_k$, $UA_{\eta,k}U=-A_{\eta,k}$, $UW_RU=W_R$, hence, writing $L_\eta(\tau,k)=\tau+A_{\eta,k}-\eta a_kW_R$ (as an operator on $L^2(S_k)$; see Section~\ref{sec:analyticity}),
\[
U\,L_\eta(\tau,k)\,U=-L_\eta(-\tau,k),\qquad \ellf\circ U=\ellf ,
\]
and since $E^\R_\eta b_\eta=L_\eta^{-1}a_k$ on the interaction subspace (Lemma~\ref{lem:geom} below), we get $L_\eta(-\tau)^{-1}=-U\,L_\eta(\tau)^{-1}U$ and hence
\begin{align*}
F_\eta(-\tau,k)&=1+V\,\ellf\bigl(L_\eta(-\tau)^{-1}a_k\bigr)
=1+V\,\ellf\bigl(-U\,L_\eta(\tau)^{-1}\,Ua_k\bigr)\\
&=1+V\,\ellf\bigl(U\,L_\eta(\tau)^{-1}a_k\bigr)
=F_\eta(\tau,k),
\end{align*}
using $Ua_k=-a_k$ and $\ellf\circ U=\ellf$; i.e.\ $F_\eta$ is even in $\tau$, and $\lambda_{-,\eta}(k)=-i\tau_\eta(|k|)$. Rotational invariance $D_\eta(\lambda,\mathcal Rk)=D_\eta(\lambda,k)$ for $\mathcal R\in SO(3)$ makes $\tau_\eta$ radial, and the $C^m$ regularity in $r$ follows from the implicit function theorem with the $C^m$ bounds of Lemma~\ref{lem:Dpert}.
\end{proof}

\section{Analyticity in the exchange coupling and the exchange shift}\label{sec:analyticity}

The observation driving this section is that, on the imaginary axis, all $\eta$-dependence of the problem is \emph{affine}: with
\begin{equation}\label{eq:LC}
L_\eta(\tau,k)=\underbrace{\tau+A_{0,k}}_{=:L_{0,\tau}}+\ \eta\,C_k,
\qquad C_k:=A_{X,k}-a_kW_R ,
\end{equation}
one has $F_\eta(\tau,k)=1+V(k)\,\ellf(L_\eta^{-1}a_k)$, where $L_\eta^{-1}a_k$ is well defined on $L^2(S_k)$ as follows.

\begin{lemma}[Geometric series on the interaction subspace]\label{lem:geom}
Let $k\in\Omega_I$ and $|\tau-\tau_0(|k|)|\le r_I$. Define $L_{0,\tau}^{-1}h:=h/(\tau+A_{0,k})$ on $L^2(S_k)$ \textup{(}bounded by Lemma~\textup{\ref{lem:gap}}, extended by zero off $S_k$\textup{)}. Then $C_k$ preserves $L^2(S_k)$,
\begin{equation}\label{eq:etastar}
\bigl\|L_{0,\tau}^{-1}C_k\bigr\|_{L^2(S_k)\to L^2(S_k)}\le
\frac{\|A_{X,k}\|_{L^\infty(S_k)}+\|a_k\|_{L^\infty}\|W_R\|_{X\to X}}{c_I}=:\frac1{\eta_*},
\end{equation}
with $\eta_*>0$ uniform over $k\in\Omega_I$ and $\tau$ as above; the operator $L_\eta(\tau,k)$ is invertible on $L^2(S_k)$ for $|\eta|<\eta_*$, and
\begin{equation}\label{eq:geomF}
F_\eta(\tau,k)=1+V(k)\sum_{n\ge0}(-\eta)^n\,
\ellf\Bigl(\bigl(L_{0,\tau}^{-1}C_k\bigr)^{n}L_{0,\tau}^{-1}a_k\Bigr),
\end{equation}
absolutely and uniformly. Moreover $E^\R_\eta(\tau,k)\,b_\eta(\tau,k)=L_\eta(\tau,k)^{-1}a_k$.
\end{lemma}

\begin{proof}
$A_{X,k}$ is a multiplication operator, hence preserves supports; the range of $a_kW_R$ lies in $L^2(S_k)$; thus $C_k:L^2(S_k)\to L^2(S_k)$ and \eqref{eq:etastar} follows from Lemmas~\ref{lem:sigma}, \ref{lem:WR}, \ref{lem:gap}. Invertibility of $L_\eta=L_{0,\tau}(\Id+\eta L_{0,\tau}^{-1}C_k)$ and \eqref{eq:geomF} are the Neumann series; the last identity is the factorization
$(\Id-\eta b_\eta W_R)^{-1}(\tau+A_{\eta,k})^{-1}a_k=[(\tau+A_{\eta,k})-\eta a_kW_R]^{-1}a_k=L_\eta^{-1}a_k$,
valid on $L^2(S_k)$, together with the bound $|\ellf(h)|\le(2\pi)^{-3}|S_k|^{1/2}\|h\|_{L^2}$.
\end{proof}

\begin{proof}[Proof of Theorem~\ref{thm:main}\,(ii)]
\emph{Analyticity.} By Lemmas~\ref{lem:gap}--\ref{lem:neumann} and Definition~\ref{def:Deta}, $D_\eta(\lambda,k)$ is jointly holomorphic in $(\lambda,\eta)$ on $U_{\sigma,k}\times\{|\eta|<\eta_1\}$, $\eta_1:=\min\{\eta_0,\ \eta_*,\ d_0/(2C_I)\}$ (the last constraint keeps the Rouch\'e condition $C_I|\eta|<d_0$ in force on the whole disk of analyticity), and Lemma~\ref{lem:Dpert} holds verbatim for complex $\eta$. The zero is then unique and simple by Rouch\'e, and the residue calculus formula
\begin{equation}\label{eq:contour}
\lambda_{\sigma,\eta}(k)=\frac1{2\pi i}\oint_{\partial U_{\sigma,k}}\lambda\,\frac{\partial_\lambda D_\eta(\lambda,k)}{D_\eta(\lambda,k)}\dd\lambda
\end{equation}
represents $\lambda_{\sigma,\eta}(k)$ as an integral of a function holomorphic in $\eta$ with uniform bounds ($|D_\eta|\ge d_0-C_I|\eta|\ge d_0/2$ on $\partial U_{\sigma,k}$); differentiation under the integral shows that $\eta\mapsto\lambda_{\sigma,\eta}(k)$ is holomorphic on $\{|\eta|<\eta_1\}$. For real $\eta$ part (i) gives $\lambda_{\sigma,\eta}=\sigma i\tau_\eta$ with $\tau_\eta$ real; hence the Taylor coefficients of $\tau_\eta$ are real and the series converges absolutely and uniformly on $I$ for $|\eta|<\eta_1$. The residues $r_{\sigma,\eta}=1/\partial_\lambda D_\eta(\lambda_{\sigma,\eta},k)$ are holomorphic in $\eta$ for the same reason (the denominator is uniformly bounded away from zero by simplicity). The error bound follows from Cauchy's estimates on the circle $|\zeta|=3\eta_1/4$ together with $|\eta|\le\eta_1/2$.

\emph{Coefficients.} Substituting the geometric series \eqref{eq:geomF} into $F_\eta(\tau_\eta(r),\eta)=0$ and solving order by order determines $\tau_X^{(n)}$ recursively as finite combinations of the absolutely convergent integrals in \eqref{eq:geomF}; the $C^{m-2}(I)$ regularity follows as in Proposition~\ref{prop:shift} below. The case $n=1$ is Proposition~\ref{prop:shift}; the case $n=2$ is Corollary~\ref{cor:second}.
\end{proof}

\begin{proposition}[First-order shift: self-energy and vertex]\label{prop:shift}
Let $r=|k|\in I$ and $L_{0,*}=\tau_0(r)+A_{0,k}$. Then
\begin{equation}\label{eq:tauX}
\tau_X(r)=-\frac{\partial_\eta F_\eta(\tau_0(r),k)\big|_{\eta=0}}{\partial_\tau F_0(\tau_0(r),k)}
=-\frac{\ellf\bigl(L_{0,*}^{-1}C_kL_{0,*}^{-1}a_k\bigr)}{\ellf\bigl(L_{0,*}^{-2}a_k\bigr)},
\end{equation}
and the decomposition \eqref{eq:tauXsplit} holds; all integrals converge absolutely and $\tau_X\in C^{m-2}(I)$.
\end{proposition}

\begin{proof}
Differentiating $F_\eta(\tau_\eta(r),k)=0$ in $\eta$ gives $\partial_\eta F+\partial_\tau F\,\partial_\eta\tau_\eta=0$. From $F_\eta=1+V\ellf(L_\eta^{-1}a_k)$ and the affine structure \eqref{eq:LC},
\[
\partial_\eta L_\eta^{-1}\big|_{\eta=0}=-L_{0,*}^{-1}C_kL_{0,*}^{-1},\qquad
\partial_\tau L_0^{-1}=-L_{0,*}^{-2},
\]
whence $\partial_\eta F|_0=-V\ellf(L_{0,*}^{-1}C_kL_{0,*}^{-1}a_k)$ and $\partial_\tau F_0=-V\ellf(L_{0,*}^{-2}a_k)$; the factor $V(k)$ cancels in the quotient. Inserting $C_k=A_{X,k}-a_kW_R$ gives \eqref{eq:tauXsplit}. The denominator is nonzero by simplicity of the Hartree zero; the numerator is finite by Lemmas~\ref{lem:WR} and \ref{lem:gap}; $k$-derivatives are handled as in Lemma~\ref{lem:neumann}.
\end{proof}

\begin{corollary}[Second-order shift]\label{cor:second}
With $\tau_X=\tau_X^{(1)}$ as above,
\begin{equation}\label{eq:tau2}
\tau_X^{(2)}(r)
=\frac{\ellf\bigl((L_{0,*}^{-1}C_k)^2L_{0,*}^{-1}a_k\bigr)
+\ellf\bigl((L_{0,*}^{-2}C_kL_{0,*}^{-1}+L_{0,*}^{-1}C_kL_{0,*}^{-2})a_k\bigr)\,\tau_X(r)
+\ellf\bigl(L_{0,*}^{-3}a_k\bigr)\,\tau_X(r)^2}
{\ellf\bigl(L_{0,*}^{-2}a_k\bigr)} .
\end{equation}
\end{corollary}

\begin{proof}
Expanding $F(\tau_0+\eta\tau_X+\eta^2\tau_X^{(2)}+O(\eta^3),\eta)=0$ at order $\eta^2$ gives
$\partial_\tau F\,\tau_X^{(2)}+\tfrac12\partial^2_\tau F\,\tau_X^2+\partial_\tau\partial_\eta F\,\tau_X+\tfrac12\partial^2_\eta F=0$ at $(\tau_0(r),0)$, where, by the affine structure \eqref{eq:LC},
\begin{gather*}
\partial_\eta^2F\big|_0=2V\ellf\bigl((L_{0,*}^{-1}C_k)^2L_{0,*}^{-1}a_k\bigr),\qquad
\partial_\tau^2F\big|_0=2V\ellf\bigl(L_{0,*}^{-3}a_k\bigr),\\
\partial_\tau\partial_\eta F\big|_0=V\ellf\bigl((L_{0,*}^{-2}C_kL_{0,*}^{-1}+L_{0,*}^{-1}C_kL_{0,*}^{-2})a_k\bigr),\qquad
\partial_\tau F_0=-V\ellf\bigl(L_{0,*}^{-2}a_k\bigr);
\end{gather*}
the factors $V$ cancel in the quotient.
\end{proof}

\begin{proposition}[Sign structure]\label{prop:signs}
For all $k\in\Omega_I$:
\begin{itemize}
\item[(a)] $\ellf\bigl(L_{0,*}^{-2}a_k\bigr)<0$; explicitly, with the marginal $\varphi_g(u)=2\pi\int_{|u|}^{\Upsilon}sG(s)\dd s$ of the radial profile $g(p)=G(|p|)$,
\begin{equation}\label{eq:Denexact}
\ellf\bigl(L_{0,*}^{-2}a_k\bigr)
=-\frac{4r^2}{(2\pi)^{3}}\int_{-\Upsilon}^{\Upsilon}\frac{\varphi_g(u)\,\bigl(\tau_0(r)+2ru\bigr)}
{\bigl[(\tau_0(r)+2ru)^2-r^4\bigr]^{2}}\dd u\ <\ 0 .
\end{equation}
\item[(b)] $\ellf\bigl(L_{0,*}^{-1}a_kW_RL_{0,*}^{-1}a_k\bigr)>0$; hence $\tau_X^{\mathrm{vertex}}<0$.
\item[(c)] If in addition $G$ is nonincreasing, then $\ellf\bigl(L_{0,*}^{-1}A_{X,k}L_{0,*}^{-1}a_k\bigr)\ge0$; hence $\tau_X^{\mathrm{self}}\ge0$.
\end{itemize}
Thus, for radial nonincreasing profiles, the static self-energy raises and the dynamic vertex lowers the plasmon frequency at first order.
\end{proposition}

\begin{proof}
(a) Shifting $p\mapsto p\pm\frac k2$ in the two terms of $a_k$ and reducing to the marginal $\varphi_g$,
$\ellf(L_{0,*}^{-2}a_k)=(2\pi)^{-3}\int\varphi_g(u)\bigl[(\tau_0+2ru+r^2)^{-2}-(\tau_0+2ru-r^2)^{-2}\bigr]\dd u$,
and the exact algebra $A^{-2}-B^{-2}=-(A-B)(A+B)/(A^2B^2)$ with $A=\tau_0+2ru+r^2$, $B=\tau_0+2ru-r^2$ gives \eqref{eq:Denexact}; positivity of the integrand uses $\tau_0+2ru-r^2\ge\delta_I>0$ on $\supp\varphi_g$ (Lemma~\ref{lem:gap} at $\eta=0$).
(b) With $b=L_{0,*}^{-1}a_k$ real,
$\ellf(bW_Rb)=(2\pi)^{-3}\langle b,Wb\rangle$ and the kernel \eqref{eq:W} has positive Fourier transform ($\widehat{|x|^{-2}}=2\pi^2/|\xi|$ in $\R^3$), so the quadratic form is positive for $b\ne0$.
(c) $A_{X,k}(p)a_k(p)=\bigl[\Sigma(p-\tfrac k2)-\Sigma(p+\tfrac k2)\bigr]\bigl[g(p-\tfrac k2)-g(p+\tfrac k2)\bigr]\ge0$ pointwise, because $\Sigma$ is radial nonincreasing whenever $G$ is (the convolution of two nonnegative radial nonincreasing functions is radial nonincreasing, by the layer-cake representation and the fact that the convolution of indicators of centered balls is radial nonincreasing), so both brackets have the sign of $-2\,u\,r$ simultaneously. Since $L_{0,*}^{-2}>0$ on $S_k$, the integral is nonnegative.
\end{proof}

\section{Exact plasmon modes}\label{sec:modes}

On the weighted space $X_w$ of \eqref{eq:Xw}, define the fiber generator
\begin{equation}\label{eq:generator}
\sG_{\eta,k}\,q:=-i\bigl(A_{\eta,k}\,q+V(k)\,a_k\,\ellf(q)-\eta\,a_k\,Wq\bigr),
\qquad D(\sG_{\eta,k})=\{q\in X_w:\ A_{\eta,k}q\in X_w\},
\end{equation}
so that \eqref{eq:fibereq} reads $\partial_tq_k=\sG_{\eta,k}q_k$. Recall from Section~\ref{sec:fiber} that $\ellf$ is bounded on $X_w$ but not closable on $L^2$, so the weight is essential for $\sG_{\eta,k}$ to be a closed operator.

\begin{lemma}[The fiber generator and its essential spectrum]\label{lem:generator}
$\sG_{\eta,k}$ is a closed, densely defined operator on $X_w$ generating the $C_0$-group of Lemma~\ref{lem:wp}. The perturbation $Bq:=-i\bigl(V(k)a_k\ellf(q)-\eta\,a_kWq\bigr)$ is compact on $X_w$, and consequently
\[
\sigma_{\mathrm{ess}}(\sG_{\eta,k})=\sigma_{\mathrm{ess}}\bigl(-iA_{\eta,k}\,\cdot\,\bigr)=i\R .
\]
\end{lemma}

\begin{proof}
The multiplication operator $-iA_{\eta,k}$ with the stated domain is closed on $X_w$ (the weight commutes with multiplication) and generates the isometric group $e^{-iA_{\eta,k}t}$; a bounded perturbation preserves closedness, the domain, and the generation property \cite{Kato}. For compactness of $B$: the Hartree part is bounded of rank one. For the exchange part, write $a_kW=a_kW\mathbf 1_{B_R}+a_kW\mathbf 1_{B_R^c}$. The multiplier identity $\widehat{|x|^{-2}}(\xi)=2\pi^2/|\xi|$ in $\R^3$ gives $W:L^2\to\dot H^1$, so $a_kW\mathbf 1_{B_R}$ maps $X_w\subset L^2$ boundedly into $H^1(S_k)$, which embeds compactly into $L^2(S_k)\subset X_w$ by Rellich's theorem. For the tail, if $p\in S_k$ and $|q|>R$ then $V(p-q)\lesssim\langle q\rangle^{-2}$, so the kernel of $a_kW\mathbf 1_{B_R^c}$, transferred to the weighted spaces, is dominated by $C|a_k(p)|\langle p\rangle^{2}\langle q\rangle^{-4}\in L^2(\dd p\dd q)$ and the tail is Hilbert--Schmidt \textup{(}cf.\ \eqref{eq:tail}\textup{)}. Since $B$ is compact, it is $\sG$-relatively compact, and relatively compact perturbations preserve the essential spectrum \cite[Theorem~IV.5.35]{Kato}. Finally, the essential spectrum of the multiplication operator is $-i$ times the essential range of $A_{\eta,k}$, which is all of $\R$: $A_{\eta,k}$ is continuous and $A_{\eta,k}(p)\to\pm\infty$ as $p\cdot k\to\pm\infty$ by Lemma~\ref{lem:sigma}.
\end{proof}

\begin{proof}[Proof of Theorem~\ref{thm:main}\,(iii)]
\emph{Existence and normalization.} An exponential solution $q=e^{\lambda t}\varphi$ satisfies
\begin{equation}\label{eq:eigen}
(\lambda+iA_{\eta,k})\varphi+iV(k)a_k\,\ellf(\varphi)-i\eta\,a_kW\varphi=0 .
\end{equation}
For $\lambda\in U_{\sigma,k}$ the right two terms are supported in $S_k$, and dividing by $\lambda+iA_{\eta,k}$ on $S_k$ (Lemma~\ref{lem:gap}),
\begin{gather*}
\varphi=-iV(k)\ellf(\varphi)\,m_\eta+i\eta\,m_\eta W_R\varphi
\ \Longrightarrow\
(\Id-i\eta K_\eta)\varphi=-iV(k)\ellf(\varphi)\,m_\eta\\
\Longrightarrow\
\varphi=-iV(k)\ellf(\varphi)\,E_\eta m_\eta .
\end{gather*}
Apply $\ellf$: either $\ellf(\varphi)=0$, and then $\varphi=0$; or the compatibility condition $1=-iV\ellf(E_\eta m_\eta)=1-D_\eta(\lambda,k)$ holds, i.e.\ $D_\eta(\lambda,k)=0$, i.e.\ $\lambda=\lambda_{\sigma,\eta}(k)$ by part~(i). The choice $\ellf(\varphi)=1$ gives $\varphi_{\sigma,\eta,k}=-iV(k)E_\eta m_\eta|_{\lambda=\lambda_{\sigma,\eta}}$, and then indeed
\[
\ellf(\varphi_{\sigma,\eta,k})=-iV\ellf(E_\eta m_\eta)=1-D_\eta(\lambda_{\sigma,\eta},k)=1 .
\]
The regularity $\varphi\in H^1$ follows from $\varphi=-iV\bigl(m_\eta+i\eta\,m_\eta W_R E_\eta m_\eta\bigr)$, the mapping property $W_R:L^2\to\dot H^1_{\mathrm{loc}}$ (Fourier multiplier $\widehat{|x|^{-2}}=2\pi^2/|\xi|$), and $m_\eta\in C^m_c$.

\emph{Uniqueness and support.} Let $(\lambda,q)$ solve \eqref{eq:eigen} with $\lambda\in U_{\sigma,k}$, $q\in D(\sG_{\eta,k})$. Off $S_k$, \eqref{eq:eigen} reads $(\lambda+iA_{\eta,k}(p))q(p)=0$. If $\Re\lambda\ne0$ the factor never vanishes ($A_{\eta,k}$ real), so $q=0$ off $S_k$. If $\lambda=i\tau$, the level set $\{\tau+A_{\eta,k}=0\}$ is a $C^{m+2}$ hypersurface of measure zero, since $|\nabla A_{\eta,k}|\ge2k_--C|\eta|>0$ (Lemma~\ref{lem:sigma}); again $q=0$ a.e.\ off $S_k$. On $S_k$ the computation above forces $q=\ellf(q)\varphi_{\sigma,\eta,k}$ and $\lambda=\lambda_{\sigma,\eta}(k)$.

\emph{Wave packets.} For $\chi\in C_c^{m-1}(\Omega_I)$, the superposition with kernel
$\widehat Q(t,p+\frac k2,p-\frac k2)=\chi(k)e^{\sigma i\tau_\eta(|k|)t}\varphi_{\sigma,\eta,k}(p)$
solves the linearized equation exactly (each fiber does, and $k\mapsto\varphi_{\sigma,\eta,k}$ is continuous in $L^2$ with uniform compact support), and by $\ellf(\varphi)=1$ its density is the Klein--Gordon-type packet
$\rho(t,x)=(2\pi)^{-3}\int\chi(k)e^{ik\cdot x+\sigma i\tau_\eta(|k|)t}\dd k$,
to which the dispersive bounds of Section~\ref{sec:dispersive} apply.
\end{proof}

\begin{remark}[Stability of an embedded eigenvalue]\label{rem:embedded}
The eigenvalues $\pm i\tau_\eta(|k|)$ are embedded in the essential spectrum $\sigma_{\mathrm{ess}}(\sG_{\eta,k})=i\R$ (Lemma~\ref{lem:generator}). Embedded eigenvalues are generically destroyed by perturbation; here they persist because the compact support of $a_k$ together with the uniform gap of Lemma~\ref{lem:gap} keeps the eigenfunction's momentum support away from the resonant set $\{\tau_\eta+A_{\eta,k}=0\}$. The compact-support hypothesis on the equilibrium is thus not a technical convenience: it is what protects the modes from Landau damping, in accordance with the threshold picture of \cite{NY1,NguyenJFA}.
\end{remark}

\begin{remark}[Residues and modes are two faces of one object]\label{rem:rankone}
Comparing with \eqref{eq:res37} below, the rank-one residue of $\sA_\eta^{-1}$ at $\lambda_{\sigma,\eta}$ is proportional to $\varphi_{\sigma,\eta,k}\otimes(\ellf\circ E_\eta(\lambda_{\sigma,\eta},k))$: the eigenfunction, the pole, and the residue are produced by one and the same Schur mechanism.
\end{remark}

\section{Green function decomposition and the sine form}\label{sec:green}

\begin{proposition}[Rank-one residues]\label{prop:res}
At $\lambda=\lambda_{\sigma,\eta}(k)$,
\begin{equation}\label{eq:res37}
\Res_{\lambda=\lambda_{\sigma,\eta}(k)}\ \sA_\eta(\lambda,k)^{-1}
=-\frac{iV(k)}{\partial_\lambda D_\eta(\lambda_{\sigma,\eta},k)}\;
E_\eta m_\eta\otimes\ellf\,E_\eta\Big|_{\lambda=\lambda_{\sigma,\eta}},
\end{equation}
a rank-one operator on $X$.
\end{proposition}

\begin{proof}
In \eqref{eq:Ainv}, $E_\eta$ and $m_\eta$ are holomorphic near $\lambda_{\sigma,\eta}$ and $D_\eta$ has a simple zero there, so $\Res D_\eta^{-1}=1/\partial_\lambda D_\eta$.
\end{proof}

Throughout the remainder of this section $\eta$ is \emph{real}, $|\eta|<\eta_0$ --- the only case needed for the Green function; reality of $A_{\eta,k}$ is used in the bounds below. The local analysis so far constructs $D_\eta$ only near the two disks. To speak of a \emph{causal} Green function, and to identify the local Laurent data with the singular part of its Laplace transform, two further ingredients are needed: a half-plane of analyticity with decay along vertical lines, and a \emph{single analytic continuation} joining the half-plane to the disks. Both are supplied by the pointwise invertibility of $\lambda+iA_{\eta,k}$ for $\Re\lambda>0$, combined with the analytic Fredholm theorem. Set
\[
M_I:=\sup_{|\eta|<\eta_0}\sup_{k\in\Omega_I}\sup_{p\in S_k}|A_{\eta,k}(p)|\le2k_+\Bigl(\Upsilon+\tfrac{k_+}{2}\Bigr)+C\eta_0,
\qquad
\Omega_k:=\{\Re\lambda>0\}\cup U_{+,k}\cup U_{-,k},
\]
and note that $\Omega_k$ is open and \emph{connected}: each disk is centered on the imaginary axis and meets the half-plane.

\begin{lemma}[Half-plane bounds with decay along vertical lines]\label{lem:halfplane}
Let $\eta$ be real, $|\eta|<\eta_0$, and $k\in\Omega_I$. The family $\lambda\mapsto m_\eta(\lambda,k)\in\mathcal B(X)$ is holomorphic on all of $\Omega_k$, given by the single formula \eqref{eq:mK} (the denominator is nonvanishing on $S_k$: by $|\lambda+iA_{\eta,k}|\ge\Re\lambda$ on the half-plane, by Lemma~\ref{lem:gap} on the disks), and for $\Re\lambda>0$,
\begin{equation}\label{eq:mdecay}
\|m_\eta(\lambda,k)\|_{L^\infty}\ \le\ \frac{\|a_k\|_{L^\infty}}{\max\{\Re\lambda,\ |\Im\lambda|-M_I\}} .
\end{equation}
Consequently there is $\Lambda_I\ge1$, independent of $\eta$ and of $k\in\Omega_I$, such that on the region
\[
\mathcal R_k:=\{\Re\lambda>\Lambda_I\}\cup\bigl\{\Re\lambda>0,\ |\Im\lambda|>2M_I+\Lambda_I\bigr\}
\]
the Neumann series for $E_\eta=(\Id-i\eta K_\eta)^{-1}$ converges with $\|E_\eta\|_{X\to X}\le2$, and
\begin{equation}\label{eq:Ddecay}
\bigl|D_\eta(\lambda,k)-1\bigr|\ \le\ \frac{C_I}{\max\{\Re\lambda,\ |\Im\lambda|-M_I\}}\ \le\ \frac12
\qquad(\lambda\in\mathcal R_k).
\end{equation}
In particular $D_\eta$ is zero-free on $\mathcal R_k$, and on each line $\{\Re\lambda=c\}$ with $c>\Lambda_I$,
\begin{equation}\label{eq:linedecay}
\bigl|D_\eta(c+i\omega,k)^{-1}-1\bigr|\ \le\ \frac{C_c}{1+|\omega|}\ \in\ L^2(\dd\omega),
\end{equation}
uniformly for $\Re\lambda\ge c$.
\end{lemma}

\begin{proof}
For $\Re\lambda>0$ and $p\in S_k$,
$|\lambda+iA_{\eta,k}(p)|\ge\max\{\Re\lambda,\,|\Im\lambda|-|A_{\eta,k}(p)|\}\ge\max\{\Re\lambda,\,|\Im\lambda|-M_I\}$,
which gives \eqref{eq:mdecay}, and the quotient is holomorphic with locally uniform bounds, hence $\mathcal B(X)$-holomorphic; on the disks the same formula is holomorphic by Lemma~\ref{lem:gap}, and the two readings agree on the overlap since the formula is one and the same. Then $\|K_\eta\|\le\|a_k\|_\infty\|W_R\|/\max\{\Re\lambda,|\Im\lambda|-M_I\}$; choosing $\Lambda_I:=\max\{1,\ 2\eta_0\|a_k\|_\infty\|W_R\|,\ 2C_I\}$ makes $|\eta|\,\|K_\eta\|\le\frac12$ on $\mathcal R_k$, and
$|D_\eta-1|\le V(k)(2\pi)^{-3}|S_k|^{1/2}\|E_\eta\|\|m_\eta\|_{L^2}$ gives \eqref{eq:Ddecay}; finally $|D_\eta^{-1}-1|\le2|D_\eta-1|$ on $\mathcal R_k$ yields \eqref{eq:linedecay}.
\end{proof}

\begin{lemma}[Meromorphic continuation and identification]\label{lem:continuation}
Let $\eta$ be real, $|\eta|<\eta_0$, and $k\in\Omega_I$. The family $K_\eta(\lambda,k)=m_\eta(\lambda,k)W_R$ is holomorphic on $\Omega_k$ with values in the compact operators on $X$, and $\Id-i\eta K_\eta$ is invertible on $\mathcal R_k$. By the analytic Fredholm theorem \cite[Theorem~VI.14]{ReedSimon}, applied on the \emph{connected} open set $\Omega_k$, the inverse $E_\eta$ extends to a finitely meromorphic family on $\Omega_k$ whose singular set $S_{\eta,k}$ is discrete and satisfies
\[
S_{\eta,k}\ \subset\ \bigl\{0<\Re\lambda\le\Lambda_I,\ |\Im\lambda|\le2M_I+\Lambda_I\bigr\}\setminus\bigl(U_{+,k}\cup U_{-,k}\bigr)
\]
(the disks are excluded because the Neumann series of Lemma~\ref{lem:neumann} converges there, and $\mathcal R_k$ by Lemma~\ref{lem:halfplane}). Consequently $D_\eta=1+iV(k)\,\ellf(E_\eta m_\eta)$ is a \emph{single} meromorphic function on $\Omega_k$ restricting to the constructions of Definition~\ref{def:Deta} on the disks and of Lemma~\ref{lem:halfplane} on $\mathcal R_k$; hence $D_\eta^{-1}$ is meromorphic on $\Omega_k$, its only singularities in $U_{+,k}\cup U_{-,k}$ are the simple poles $\lambda_{\pm,\eta}(k)$, all its other singularities lie in the bounded set $Q_k:=\{0<\Re\lambda\le\Lambda_I,\ |\Im\lambda|\le2M_I+\Lambda_I\}\setminus(U_{+,k}\cup U_{-,k})$ (bounded but not closed: its closure may meet the imaginary axis, and no claim is made that the singularities stay away from it), and the principal parts at $\lambda_{\sigma,\eta}(k)$ of the same function that represents the Laplace transform on $\{\Re\lambda>\Lambda_I\}$ are $r_{\sigma,\eta}(k)\,(\lambda-\lambda_{\sigma,\eta}(k))^{-1}$, as computed in \eqref{eq:laurent}.
\end{lemma}

\begin{proof}
Compactness of $W_R$ is Lemma~\ref{lem:WR}; holomorphy of $m_\eta$ on $\Omega_k$ and invertibility of $\Id-i\eta K_\eta$ on $\mathcal R_k$ are Lemma~\ref{lem:halfplane}. The analytic Fredholm alternative on the connected set $\Omega_k$ gives the finitely meromorphic inverse; since a bounded inverse is unique wherever it exists, the extension agrees with the Neumann inverses on the disks and on $\mathcal R_k$, which locates $S_{\eta,k}$ as claimed. Then $\ellf(E_\eta m_\eta)$, hence $D_\eta$, is meromorphic on $\Omega_k$ and agrees with the earlier local definitions; possible zeros of $D_\eta$ outside the disks cannot occur on $\mathcal R_k$ by \eqref{eq:Ddecay}, nor in the disks away from $\lambda_{\pm,\eta}$ by Theorem~\ref{thm:main}\,(i), so all singularities of $D_\eta^{-1}$ apart from $\lambda_{\pm,\eta}$ lie in $Q_k$. The identification of the principal parts is now immediate from connectedness and the identity theorem: there is one meromorphic function, and its Laurent expansion at $\lambda_{\sigma,\eta}$ was computed in \eqref{eq:laurent}.
\end{proof}

\begin{proposition}[Green function; sine form]\label{prop:green}
Let $\eta$ be real, $|\eta|<\eta_0$, and $k\in\Omega_I$.
\begin{itemize}
\item[\textup{(a)}] \textup{(Causal inverse transform.)} By \eqref{eq:linedecay} the function $\lambda\mapsto D_\eta(\lambda,k)^{-1}-1$ belongs to the Hardy space $H^2$ of every half-plane $\{\Re\lambda>c\}$, $c>\Lambda_I$. By the Paley--Wiener theorem for $H^2$ of a half-plane there is a unique measurable $g_\eta(\cdot,k)$, supported in $[0,\infty)$, with $e^{-ct}g_\eta(\cdot,k)\in L^2(\R)$ and
\[
\mathcal L[g_\eta](\lambda,k)=D_\eta(\lambda,k)^{-1}-1\qquad(\Re\lambda>c);
\]
$g_\eta$ does not depend on the abscissa $c$. Define the causal density Green function by
\begin{equation}\label{eq:Gdef}
\widehat G_\eta(t,k):=\delta(t)+g_\eta(t,k).
\end{equation}
\item[\textup{(b)}] \textup{(Principal parts.)} Write $\widetilde G_\eta$ for the unique meromorphic continuation of $D_\eta(\cdot,k)^{-1}$ from $\{\Re\lambda>\Lambda_I\}$ to $\Omega_k$ (Lemma~\ref{lem:continuation}). Then, on $U_{+,k}\cup U_{-,k}$,
\begin{equation}\label{eq:laurent}
\widetilde G_\eta(\lambda,k)=\sum_{\sigma=\pm}\frac{r_{\sigma,\eta}(k)}{\lambda-\lambda_{\sigma,\eta}(k)}+H_\eta(\lambda,k),
\qquad r_{\sigma,\eta}(k)=\frac{1}{\partial_\lambda D_\eta(\lambda_{\sigma,\eta}(k),k)},
\end{equation}
with $H_\eta$ holomorphic and uniformly bounded. The residues are purely imaginary,
\begin{equation}\label{eq:residuesine}
r_{\sigma,\eta}(k)=\frac{\sigma\,i}{\partial_\tau F_\eta(\tau_\eta(|k|),k)},\qquad r_{-,\eta}=-r_{+,\eta},
\end{equation}
and the explicit \emph{pole part} satisfies
\begin{equation}\label{eq:sine}
\begin{gathered}
\widehat P_\eta(t,k):=-\frac{2}{\partial_\tau F_\eta(\tau_\eta(|k|),k)}\,\sin\bigl(\tau_\eta(|k|)t\bigr)\,\mathbf 1_{t\ge0},\\
\mathcal L[\widehat P_\eta](\lambda,k)=\sum_{\sigma=\pm}\frac{r_{\sigma,\eta}(k)}{\lambda-\lambda_{\sigma,\eta}(k)}
\qquad(\Re\lambda>0).
\end{gathered}
\end{equation}
\item[\textup{(c)}] \textup{(Remainder.)} The remainder $\widehat G^{\mathrm{rem}}_\eta:=\widehat G_\eta-\delta-\widehat P_\eta$ has Laplace transform
\begin{equation}\label{eq:remtransform}
\mathcal L[\widehat G^{\mathrm{rem}}_\eta](\lambda,k)
=\widetilde G_\eta(\lambda,k)-1-\sum_{\sigma=\pm}\frac{r_{\sigma,\eta}(k)}{\lambda-\lambda_{\sigma,\eta}(k)}
\qquad(\Re\lambda>\Lambda_I),
\end{equation}
and the right-hand side --- one meromorphic function on $\Omega_k$ --- is holomorphic on $\{\Re\lambda>\Lambda_I\}\cup U_{+,k}\cup U_{-,k}$ and uniformly bounded on the two disks; all its other possible singularities lie in the fixed bounded set $Q_k$ of Lemma~\ref{lem:continuation}, about which no assertion is made. In particular the decomposition of Theorem~\textup{\ref{thm:main}\,(iv)} holds.
\end{itemize}
\end{proposition}

\begin{proof}
(a) By Lemma~\ref{lem:halfplane}, $D_\eta^{-1}-1$ is holomorphic on $\{\Re\lambda>\Lambda_I\}$ and, by \eqref{eq:linedecay}, $\sup_{c'\ge c}\|D_\eta(c'+i\cdot,k)^{-1}-1\|_{L^2(\dd\omega)}<\infty$ for each fixed $c>\Lambda_I$: this is membership in $H^2(\{\Re\lambda>c\})$. The Paley--Wiener representation theorem for $H^2$ of a half-plane \cite[Theorem~19.2]{Rudin} yields a unique inverse transform $g_\eta$ with $e^{-ct}g_\eta\in L^2$ supported in $t\ge0$; applying uniqueness to two abscissas $\Lambda_I<c<c'$ (the corresponding inverse transforms have the same Laplace transform on $\{\Re\lambda>c'\}$) shows that $g_\eta$ does not depend on $c$.

(b) \eqref{eq:laurent} is the Laurent expansion of the continued function $\widetilde G_\eta$ at the two simple zeros of $D_\eta$, legitimate because Lemma~\ref{lem:continuation} identifies $\widetilde G_\eta$ on the disks with the local function $D_\eta^{-1}$ of Sections~\ref{sec:schur}--\ref{sec:persistence}; boundedness of $H_\eta$ on the (shrunken) disks is uniform because $\partial_\lambda D_\eta$ is uniformly bounded away from zero and the pole set is compact. For \eqref{eq:residuesine}: on the axis $D_\eta(i\tau,k)=F_\eta(\tau,k)$, so $\partial_\lambda D_\eta|_{\lambda=\sigma i\tau_\eta}=-i\,\partial_\tau F_\eta(\sigma\tau_\eta,k)$, and $F_\eta$ is real and even in $\tau$ (Section~\ref{sec:persistence}), so $\partial_\tau F_\eta(-\tau_\eta,k)=-\partial_\tau F_\eta(\tau_\eta,k)$; \eqref{eq:residuesine} follows, and
$r_+e^{i\tau_\eta t}+r_-e^{-i\tau_\eta t}=2i\,r_+\sin(\tau_\eta t)=-\frac{2}{\partial_\tau F_\eta}\sin(\tau_\eta t)$
gives the second identity in \eqref{eq:sine} by the elementary formulas $\mathcal L[\mathbf 1_{t\ge0}e^{\lambda_*t}](\lambda)=(\lambda-\lambda_*)^{-1}$, $\Re\lambda>0$.

(c) For $\Re\lambda>\Lambda_I$ the identity \eqref{eq:remtransform} holds by (a), \eqref{eq:Gdef}, and \eqref{eq:sine}, since there $\widetilde G_\eta=D_\eta^{-1}$. The right-hand side is a single meromorphic function on the connected set $\Omega_k$ (Lemma~\ref{lem:continuation}); on the disks it equals $H_\eta-1$ by \eqref{eq:laurent} --- the subtracted principal parts cancel the Laurent singular parts of the \emph{same} continued function, which is exactly the identification supplied by Lemma~\ref{lem:continuation} --- hence it is holomorphic there with uniform bounds, and its remaining singularities lie in $Q_k$. The logical chain is: Laplace transform on the half-plane (a) $\to$ unique meromorphic continuation on the connected domain $\Omega_k$ (Lemma~\ref{lem:continuation}) $\to$ Laurent principal parts at $\lambda_{\pm,\eta}$ (b) $\to$ inverse Laplace transform of the principal parts \eqref{eq:sine}. No contour deformation across the imaginary axis is used at any point.
\end{proof}

\begin{remark}[General data: collective pole versus free streaming]\label{rem:generaldata}
For initial data $q_{k,0}$ supported in $S_k$ (more generally, in any compact set where the gap \eqref{eq:gap} holds), the numerator $\ellf(E_\eta s_{k,0})$ in \eqref{eq:rhotilde} is holomorphic on the disks and the density resolvent is meromorphic there; consistently with Proposition~\ref{prop:schur}\,(b), its possible poles are contained in $\{\lambda_{+,\eta},\lambda_{-,\eta}\}$ --- the operator-valued response has genuine rank-one poles at both (Proposition~\ref{prop:schur}\,(a)), while a fixed datum may cancel either residue. For data $q_{k,0}\in X_w$ whose momentum support meets the resonant set $\{\tau_\eta(|k|)+A_{\eta,k}(p)=0\}$, the free factor $s_{k,0}(\lambda)$ is holomorphic only in $\{\Re\lambda>0\}$ and does not continue meromorphically across the axis. If the data are in addition H\"older continuous --- for instance $q_{k,0}\in C^1_c(\R^3)$ --- then the boundary values on the axis exist and are H\"older continuous by classical Plemelj theory, and the density decomposes into the coherent, undamped pole part \eqref{eq:sine} plus a \emph{bounded, non-meromorphic branch contribution} (free-streaming phase mixing at the same frequency), blow-up singularities occurring only at the zeros of $D_\eta$. For data with $X_w$ regularity alone such pointwise boundary values may fail --- the relevant Cauchy-type transform does not preserve $L^1$-type classes --- and the branch contribution exists only in a weak sense; a quantitative version would require transversal Sobolev or Besov regularity of the data on the resonant set. In all cases, meromorphic continuation across the axis is guaranteed for gap-localized data, as in Proposition~\ref{prop:schur}\,(b), and this distinction disappears for sources localized in the interaction region.
\end{remark}

\begin{remark}[What is claimed, and what is not]\label{rem:noremdecay}
Proposition~\ref{prop:green} is a statement about the \emph{analytic structure of the Laplace transform}: the sine decomposition isolates the principal parts of the unique meromorphic continuation of $\mathcal L[\widehat G_\eta]$ at the two plasmon poles. It is \emph{not} a long-time asymptotic expansion of the density: no pointwise-in-time decay (indeed, no pointwise bound whatsoever) is claimed for the continuous component $\widehat G^{\mathrm{rem}}_\eta$ with the physical Coulomb exchange, and consequently no claim is made that $\widehat P_\eta$ dominates $\widehat G_\eta$ as $t\to\infty$. Such control requires boundary-value estimates for $E_\eta(i\omega+0,k)$, $m_\eta(i\omega+0,k)$ on the particle--hole continuum, where the gap fails; this is precisely the open problem isolated in Section~\ref{sec:conditional}, whose conditional Proposition~\ref{prop:conditional} converts the missing boundary regularity \eqref{eq:conditional} into $\langle t\rangle^{-N}$ decay of the continuous component. The smooth-kernel machinery of \cite{NY2} does not apply to $|p-q|^{-2}$ directly.
\end{remark}

Propositions~\ref{prop:res} and~\ref{prop:green} together prove Theorem~\ref{thm:main}\,(iv).

\section{Dispersive and Strichartz estimates for the plasmon channel}\label{sec:dispersive}

\begin{lemma}[Uniformly nondegenerate Hessian]\label{lem:hessian}
After shrinking $\eta_0$, for $|\eta|<\eta_0$ and $r\in I$,
\begin{equation}\label{eq:hessian}
\tau_\eta''(r)\ge c_I>0,\qquad \frac{\tau_\eta'(r)}{r}\ge c_I>0 .
\end{equation}
Consequently the Hessian of $k\mapsto\tau_\eta(|k|)$ is uniformly nondegenerate on $\Omega_I$, with radial eigenvalue $\tau_\eta''$ and the double tangential eigenvalue $\tau_\eta'/|k|$.
\end{lemma}

\begin{proof}
At $\eta=0$ this is \eqref{eq:inputbounds}; the $C^2$-in-$r$ convergence $\tau_\eta\to\tau_0$ furnished by Theorem~\ref{thm:main}\,(ii) (or by the uniform $C^m$ bounds in the implicit function theorem) preserves the bounds for small $|\eta|$.
\end{proof}

\begin{proof}[Proof of Theorem~\ref{thm:main}\,(v)]
Write the kernel as an oscillatory integral with phase $\Phi_{t,x}(k)=k\cdot x+\sigma t\tau_\eta(|k|)$ and amplitude $\chi r_{\sigma,\eta}\in C_c^{m-1}(\Omega_I)$ (regularity of the residues follows from \eqref{eq:laurent} and part~(ii)). Where $\nabla\Phi\ne0$, integrate by parts; where stationary points exist, apply the three-dimensional stationary phase theorem with the uniform bounds of Lemma~\ref{lem:hessian}; since $m-1\ge5$, the amplitude has enough derivatives, and
$\|K_{\sigma,\eta}(t,\cdot)\|_{L^\infty_x}\le Ct^{-3/2}$ for $t\ge1$, uniformly in $\eta,\sigma$. The $L^2\to L^2$ bound is Plancherel; Riesz--Thorin interpolation gives the dispersive family. For the Strichartz estimates set $U(t)=e^{\sigma it\tau_\eta(|D|)}\widetilde\chi(D)$; then $\sup_t\|U(t)\|_{L^2\to L^2}\le C$ and $U(t)U(s)^*=e^{\sigma i(t-s)\tau_\eta(|D|)}|\widetilde\chi|^2(D)$ has kernel bounded by $C\min(1,|t-s|^{-3/2})$ in $L^1\to L^\infty$ (the $|t-s|\le1$ regime by the trivial $L^1_k$ bound of the amplitude). Since the decay exponent $3/2$ exceeds $1$, the abstract endpoint Strichartz theorem of Keel--Tao \cite{KeelTao} applies and yields the full admissible range $\frac1q+\frac3{2p}=\frac34$, $q\ge2$, including $(q,p)=(2,6)$.
\end{proof}

\section{First-order Ward-type cancellation of the plasma gap}\label{sec:ward}

The two components of \eqref{eq:tauXsplit} are individually of order one on the band (they converge to nonzero limits as $r\to0$; see Section~\ref{sec:numerics}); their sum, however, is quadratically small. The mechanism is the following identity.

\begin{lemma}[The self-energy increment is the Coulomb potential of $a_k$]\label{lem:AXident}
$A_{X,k}=W a_k$. Consequently, with $L=L_{0,*}$ and $\operatorname{Diff}(r):=\ellf(L^{-1}A_{X,k}L^{-1}a_k)-\ellf(L^{-1}a_kW_RL^{-1}a_k)$,
\begin{equation}\label{eq:diffcombined}
\operatorname{Diff}(r)=\frac{2r^2}{(2\pi)^6}\iint_{\R^3\times\R^3}
V(p-q)\,a_k(p)\,a_k(q)\,\frac{(p_1-q_1)^2}{L(p)^2L(q)^2}\dd p\dd q ,
\qquad k=re_1 ,
\end{equation}
an absolutely convergent integral with bounded integrand \textup{(}the factor $(p_1-q_1)^2$ cancels the Coulomb singularity, since $V(p-q)(p_1-q_1)^2\le4\pi$\textup{)}.
\end{lemma}

\begin{proof}
Shifting the integration variable in \eqref{eq:selfenergy},
\[
A_{X,k}(p)=\Sigma\Bigl(p-\tfrac k2\Bigr)-\Sigma\Bigl(p+\tfrac k2\Bigr)
=\frac{1}{(2\pi)^3}\int V(p-q)\Bigl[g\Bigl(q-\tfrac k2\Bigr)-g\Bigl(q+\tfrac k2\Bigr)\Bigr]\dd q=(Wa_k)(p).
\]
Hence both terms of $\operatorname{Diff}$ are quadratic forms with the same kernel:
\[
\operatorname{Diff}(r)=\frac{1}{(2\pi)^6}\iint V(p-q)\,a_k(p)a_k(q)\Bigl[\frac1{L(p)^2}-\frac1{L(p)L(q)}\Bigr]\dd p\dd q .
\]
Symmetrizing in $(p,q)$ and using $\frac12\bigl[\frac1{L_p^2}+\frac1{L_q^2}-\frac2{L_pL_q}\bigr]=\frac{(L_p-L_q)^2}{2L_p^2L_q^2}$ with $L(p)-L(q)=2r(p_1-q_1)$ gives \eqref{eq:diffcombined}. (On $S_k$, $W$ and $W_R$ coincide when applied to functions supported in $S_k$.)
\end{proof}

\begin{proof}[Proof of Theorem~\ref{thm:ward}]
All formulas in \eqref{eq:tauXsplit} and \eqref{eq:Denexact} make sense for every $r\in(0,k_+]$: the gap $\delta(r)=\tau_0(r)-(2r\Upsilon+r^2)$ is continuous and positive on $(0,k_+]$ with $\delta(r)\to\omega_p>0$ as $r\to0$ by Assumption~\ref{ass:input}, so $\delta_*:=\inf_{(0,k_+]}\delta>0$, and $\tau_0$ is bounded on $(0,k_+]$. By Proposition~\ref{prop:signs}\,(a) and these bounds there are $c_1,c_2>0$ with
\begin{equation}\label{eq:denorder}
c_1r^2\ \le\ \bigl|\ellf(L^{-2}a_k)\bigr|\ \le\ c_2r^2,\qquad 0<r\le k_+ .
\end{equation}
On the other hand, by Lemma~\ref{lem:AXident}, $V(p-q)(p_1-q_1)^2\le4\pi$ and $\|a_k\|_{L^1}\le r\|\nabla g\|_{L^1}$ (from $a_k(p)=-\int_{-1/2}^{1/2}k\cdot\nabla g(p+sk)\dd s$),
\begin{equation}\label{eq:difforder}
|\operatorname{Diff}(r)|\le\frac{2r^2}{(2\pi)^6}\cdot\frac{4\pi}{\delta_*^4}\,\|a_k\|_{L^1}^2\le C\,r^4 .
\end{equation}
Since $\tau_X=-\operatorname{Diff}/\ellf(L^{-2}a_k)$ by Proposition~\ref{prop:shift} and Lemma~\ref{lem:AXident}, \eqref{eq:denorder}--\eqref{eq:difforder} give $|\tau_X(r)|\le Cr^2$.

For the limit, $a_k/r\to-\partial_1g$ pointwise with the uniform domination $|a_k|/r\le\|\nabla g\|_\infty\mathbf 1_{B_{\Upsilon+1}}$, and $L(p)=\tau_0(r)+2rp_1\to\omega_p$ uniformly on compact sets; dominated convergence in \eqref{eq:diffcombined} and in \eqref{eq:Denexact} yields
\begin{gather*}
\frac{\operatorname{Diff}(r)}{r^4}\longrightarrow \frac{2}{(2\pi)^6\,\omega_p^4}\iint V(p-q)\,\partial_1g(p)\partial_1g(q)(p_1-q_1)^2\dd p\dd q,\\
\frac{\bigl|\ellf(L^{-2}a_k)\bigr|}{r^2}\longrightarrow\frac{4}{(2\pi)^3\,\omega_p^{3}}\int_\R\varphi_g(u)\dd u=\frac{4\pi^2}{(2\pi)^3\,\omega_p},
\end{gather*}
where we used $\int_\R\varphi_g=\int_{\R^3}g=(2\pi)^3\rho_g=\pi^2\omega_p^2$. Taking the quotient (and the sign from $\tau_X=\operatorname{Diff}/|\ellf(L^{-2}a_k)|$) gives the first expression in \eqref{eq:betaX} after inserting $V=4\pi|\cdot|^{-2}$. The second expression follows from $(p_1-q_1)^2/|p-q|^2=1-|p_\perp-q_\perp|^2/|p-q|^2$ and $\iint\partial_1g(p)\partial_1g(q)\dd p\dd q=(\int\partial_1g)^2=0$. The final statement of the theorem is immediate from $|\tau_X(r)|\le Cr^2$ and the existence of the limit $\beta_X$.
\end{proof}

\begin{remark}\label{rem:wardmeaning}
Theorem~\ref{thm:ward} settles, at the level of the first-order dispersion coefficient, the expected exchange invariance of the plasma gap: the physically famous statement that the plasma frequency at $k=0$ is protected while exchange renormalizes the \emph{stiffness} of the plasmon branch. What remains open is the corresponding statement at the level of the full resolvent uniformly down to $k=0$, where the Coulomb factor $V(k)\sim|k|^{-2}$ must be cancelled against the density response through the dynamic exchange vertex; see Section~\ref{sec:conclusion}.
\end{remark}

\section{Numerical illustration}\label{sec:numerics}

We evaluate the objects of Theorems~\ref{thm:main}--\ref{thm:ward} for the smoothed Fermi ball
\[
g(p)=G(|p|),\qquad G=1\ \text{on }[0,\alpha\Upsilon],\qquad
G(s)=1-S_9\Bigl(\tfrac{s-\alpha\Upsilon}{(1-\alpha)\Upsilon}\Bigr)\ \text{on }[\alpha\Upsilon,\Upsilon],
\]
with $\Upsilon=1$, $\alpha=\tfrac12$, and the $C^9$ generalized smoothstep of degree $19$,
\[
S_9(x)=x^{10}\sum_{n=0}^{9}\binom{9+n}{n}\binom{19}{9-n}(-x)^n ,
\]
whose derivatives of orders $1,\dots,9$ vanish at both endpoints. Consequently $f(e)=G(\sqrt e)\in C^{9}([0,\infty))$ (the junctions lie at $e=\alpha^2\Upsilon^2$ and $e=\Upsilon^2$, away from $e=0$, where $f$ is constant), and $G$ vanishes to the exact order $10$ at $s=\Upsilon$; hence \textup{(G1)} holds with $m=6$ and \textup{(G2)} holds with $n_1=10$, so this equilibrium genuinely satisfies the hypotheses of Theorems~\ref{thm:main}--\ref{thm:ward}. The Hartree branch is computed from the one-dimensional marginal representation
\[
F_0(\tau,r)=1-\frac1{\pi^2}\int_{-\Upsilon}^{\Upsilon}\frac{\varphi_g(u)}{(\tau+2ru)^2-r^4}\dd u,
\qquad \varphi_g(u)=2\pi\int_{|u|}^{\Upsilon}sG(s)\dd s,
\]
by adaptive quadrature and root finding; the survival threshold solves $\Phi(\kappa_0)=0$ with
$\Phi(r)=1-(4\pi^2r^2)^{-1}\int\varphi_g(u)\,[(\Upsilon+u)(\Upsilon+r+u)]^{-1}\dd u$, giving $\kappa_0\approx0.23251$. The two-dimensional integrals \eqref{eq:Denexact} and the four-dimensional integrals \eqref{eq:diffcombined} and $\ellf(L^{-1}a_kW_RL^{-1}a_k)$ (reduced to four dimensions by the angular average $\int_0^{2\pi}(A-B\cos\theta)^{-1}\dd\theta=2\pi(A^2-B^2)^{-1/2}$) are evaluated on staggered Gauss--Legendre grids with $112\times96$ (outer) and $120\times104$ (inner) nodes; comparison with the refined $128\times112$ and $136\times120$ grids \emph{at every sample point} certifies that the combined form \eqref{eq:diffcombined}, whose integrand is bounded, is converged to relative accuracy $\le5\times10^{-5}$ on the computed range, and $\tau_X^{\mathrm{self}}$ is reconstructed as $\tau_X^{\mathrm{vertex}}$ plus the combined difference to avoid cancellation errors --- the net shift is \emph{never} computed as a direct difference of separately quadratured $O(10^{-1})$ terms. A single execution of the Python script included as an ancillary file with this submission (\texttt{anc/numerics\_tauX.py}) reproduces every number quoted in this section, the sweep data, the quadrature convergence table, and Figure~\ref{fig:numerics}, and enforces the checks below by assertions.

The computation passes the following independent checks: $\tau_0(0^+)^2=8\pi\rho_g$ to six digits; $\tau_0$ is convex and increasing and meets the continuum edge at $\kappa_0$ (cf.\ \cite[Theorem~2.6]{NY1}); the three sign predictions of Proposition~\ref{prop:signs} hold at every grid point; the small-$r$ limits match the closed forms of Theorem~\ref{thm:ward} to four digits, namely $|\ellf(L^{-2}a_k)|/r^2\to1/(2\pi\omega_p)=0.3732$ and $\tau_X/r^2\to\beta_X=-0.1089$ against the independent evaluation of \eqref{eq:betaX}; and the fitted log--log slope of $|\tau_X|$ at small $r$ is $1.997$.

\begin{figure}[t]
\centering
\includegraphics[width=\textwidth]{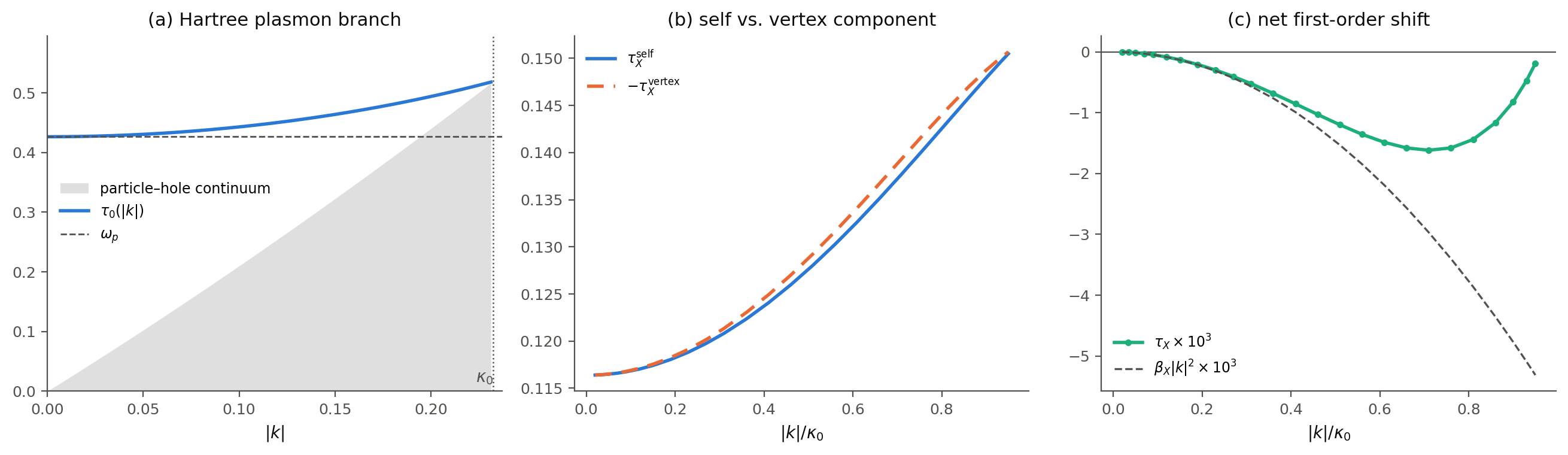}
\caption{$C^9$ smoothed Fermi ball, $\Upsilon=1$, $\alpha=\frac12$. (a) The Hartree plasmon branch $\tau_0(|k|)$ above the particle--hole continuum $|A_{0,k}|\le2|k|\Upsilon+|k|^2$; the dashed line is the plasma gap $\omega_p=\sqrt{8\pi\rho_g}$. (b) The first-order components $\tau_X^{\mathrm{self}}$ and $-\tau_X^{\mathrm{vertex}}$: the curves nearly coincide, illustrating the cancellation of Theorem~\ref{thm:ward}; each is of order $10^{-1}$ while their sum is of order $10^{-3}$. (c) The net shift $\tau_X\times10^3$ on the computed range $[0.02\,\kappa_0,\,0.95\,\kappa_0]$, with the quadratic law $\beta_X|k|^2$ (dashed) at small wave numbers; the net shift is negative (softening) throughout, most negative near $0.71\,\kappa_0$, and returns towards zero as the threshold is approached.}
\label{fig:numerics}
\end{figure}

The findings, for this model equilibrium and on the computed range $r/\kappa_0\in[0.02,\,0.95]$, are as follows (Figure~\ref{fig:numerics}). The self-energy and vertex components obey $\tau_X^{\mathrm{self}}\in[0.116,\,0.151]$ and $\tau_X^{\mathrm{vertex}}\in[-0.151,\,-0.116]$ and cancel to within $1.2\%$ everywhere (maximal ratio $|\tau_X|/\tau_X^{\mathrm{self}}=1.17\%$, at $0.71\,\kappa_0$). The net shift follows the quadratic law $\tau_X\simeq\beta_Xr^2$ with $\beta_X=-0.1089$ up to $r\approx0.4\kappa_0$; compared with the Hartree stiffness $\tau_0(r)\simeq\omega_p+\beta_0r^2$, $\beta_0=1.641$, the exchange softens the plasmon dispersion by $\eta\,\beta_X/\beta_0\approx-6.6\%\,\eta$. The net shift attains its most negative value $\approx-1.62\times10^{-3}$ near $r\approx0.71\kappa_0$ and then increases again towards zero as the survival threshold is approached, remaining negative on the computed range --- with the caveat that near the threshold the admissible coupling range $\eta_0(I)$ shrinks, so the first-order coefficient retains its meaning while the range of $\eta$ it governs narrows; we therefore do not attach significance to the precise behavior at the upper end of the range. In particular, retaining only $\tau_X^{\mathrm{self}}$ --- the ``dispersion-correction-only'' approximation --- mispredicts even the sign of the net shift throughout the computed range, quantifying Remark~\ref{rem:meaning}.

\section{A conditional criterion for the continuous component}\label{sec:conditional}

\begin{proposition}\label{prop:conditional}
Let $N\ge1$. Suppose that, for $k\in\Omega_I$, the function $\mathcal L[\widehat G^{\mathrm{rem}}_\eta](\cdot,k)$ of \eqref{eq:remtransform} --- defined a priori on $\{\Re\lambda>\Lambda_I\}\cup U_{+,k}\cup U_{-,k}$ --- extends continuously to the closed right half-plane, with boundary value
\[
\mathcal H_\eta(\omega,k)=\lim_{\gamma\downarrow0}\Bigl[D_\eta(\gamma+i\omega,k)^{-1}-1-\sum_{\sigma=\pm}\frac{r_{\sigma,\eta}(k)}{\gamma+i\omega-\lambda_{\sigma,\eta}(k)}\Bigr] .
\]
If, in addition, for some cutoff $\chi\in C_c^\infty(\Omega_I)$,
\begin{equation}\label{eq:conditional}
\sup_{k\in\Omega_I}\ \sum_{j=0}^{N}\bigl\|\partial_\omega^j\bigl(\chi(k)\mathcal H_\eta(\omega,k)\bigr)\bigr\|_{L^1_\omega(\R)}\ \le\ C_N ,
\end{equation}
then $\sup_{k\in\Omega_I}|\chi(k)\widehat G^{\mathrm{rem}}_\eta(t,k)|\le C_N'\langle t\rangle^{-N}$.
\end{proposition}

\begin{proof}
For $t\ge1$ write $\chi\widehat G^{\mathrm{rem}}_\eta(t,k)=\frac1{2\pi}\int_\R e^{i\omega t}\chi(k)\mathcal H_\eta(\omega,k)\dd\omega$ and integrate by parts $N$ times in $\omega$; the boundary terms vanish by the $W^{N,1}$ decay at infinity, and \eqref{eq:conditional} gives $t^{-N}$. For $0\le t\le1$ use the $j=0$ bound.
\end{proof}

\begin{remark}\label{rem:conditionalcoulomb}
For a smooth screened exchange kernel, \eqref{eq:conditional} is expected to follow from weighted momentum estimates of the type developed in \cite[Sections~3--4]{NY2}. For the physical Coulomb exchange the operator $W_R$ itself is bounded (Lemma~\ref{lem:WR}), but $\omega$- and $k$-derivatives of $E_\eta(i\omega+0,k)$ on the particle--hole continuum hit the non-integrable derivative singularity of $|p-q|^{-2}$; establishing \eqref{eq:conditional} in this setting appears to require Hardy--Littlewood--Sobolev-type bounds after moving derivatives onto the exchange kernel, or a limiting absorption principle in fractional Sobolev spaces. We leave this for future work.
\end{remark}

\section{Concluding remarks and open problems}\label{sec:conclusion}

The logic established in this paper is modular: the physical Coulomb exchange is a weakly singular compact operator on bounded relative-momentum regions; on compact subthreshold bands the Hartree plasmons are uniformly separated from the particle--hole continuum; the small exchange block is inverted by a Neumann series and the Coulomb--Hartree density channel is eliminated by a Feshbach--Schur complement; the resulting dielectric function is a uniformly $C^m$, jointly holomorphic small deformation of the Hartree one, so Rouch\'e's theorem, real-analytic implicit function arguments, and parity give persistence, pure imaginarity, and $\pm$ symmetry; the affine dependence on $\eta$ upgrades the perturbation theory to a convergent series with explicit coefficients; the poles are realized by closed-form eigenfunctions; Laurent expansion and stationary phase produce the sine-form Green function, the $t^{-3/2}$ dispersion, and Strichartz estimates; and the identity $A_{X,k}=Wa_k$ yields the first-order Ward-type cancellation with an explicit stiffness correction.

Four problems remain open, and we state them in the order in which they seem approachable.
\begin{itemize}
\item[1.] \emph{Uniformity down to $k=0$.} Theorem~\ref{thm:ward} proves the exchange invariance of the plasma gap at the level of the first-order coefficient, with the exact rate $O(r^2)$. Extending this to a uniform description of the resolvent on $0<|k|\le k_+$ requires tracking the cancellation between $V(k)\sim|k|^{-2}$ and the $|k|^2$-smallness of the density response through the dynamic exchange vertex --- a Ward-type identity at the level of operators rather than coefficients. The geometric series of Lemma~\ref{lem:geom} organizes this order by order and seems a natural starting point.
\item[2.] \emph{The threshold $|k|\to\kappa_0$.} The isolated pole meets the continuum; Rouch\'e perturbation is insufficient, and a Plemelj boundary-value and resonance analysis for the operator family $E_\eta m_\eta$ is needed, together with the first-order exchange correction of the survival threshold $\kappa_\eta$ and of the Landau damping rate beyond it (cf.\ \cite[Theorem~2.9]{NY1}, \cite{NguyenJFA}).
\item[3.] \emph{The continuous component.} Establish \eqref{eq:conditional} for the physical Coulomb exchange (Remark~\ref{rem:conditionalcoulomb}), replacing the smooth-kernel estimates of \cite{NY2} by fractional-regularity boundary-value estimates.
\item[4.] \emph{Nonlinearity.} The plasmon component disperses only at the Klein--Gordon rate $t^{-3/2}$ and carries the low-frequency Coulomb singularity; a nonlinear theory on the band will require momentum-dependent echo estimates in the presence of exchange (cf.\ \cite{NY2}) and a normal form for plasmon--plasmon interactions. The exact modes of Theorem~\ref{thm:main}\,(iii), with their explicit eigenfunctions and unit-density normalization, and the Strichartz estimates of Theorem~\ref{thm:main}\,(v) are intended as inputs for this program.
\end{itemize}

\appendix

\section{Fourier conventions and comparison with \texorpdfstring{\cite{NY1,NY2}}{Nguyen--You}}\label{app:conventions}

The papers \cite{NY1,NY2} evaluate two-variable kernels at $(k-p,p)$ and write the transported profile difference as $a_{k-p,p}=g(k-p)-g(p)$, while we use the symmetric coordinates $(p+\frac k2,p-\frac k2)$ and $a_k(p)=g(p-\frac k2)-g(p+\frac k2)$. The substitution $p\mapsto p+\frac k2$, together with the reflection $p\mapsto-p$ under which $a_k$ is odd and $A_{0,k}$ is odd, identifies the two descriptions; the apparent sign discrepancies of the factors $i$ in the dielectric functions are artifacts of this change of variables. The physically meaningful objects --- the zero set of $D_\eta$, the multiplicity of the zeros, and the residue operators --- are independent of the convention.

There is also a normalization dictionary: \cite{NY1} takes $\widehat w(k)=|k|^{-2}$ and unnormalized momentum integrals $\int\dd p$, whereas we take $V(k)=4\pi|k|^{-2}$ and $\ellf=(2\pi)^{-3}\int\dd p$. The reduction is \emph{exact}: with $\mu:=g/(2\pi^2)$ our $D_0$ of \eqref{eq:D0} coincides with the dielectric function of \cite{NY1} for $\widehat w(k)=|k|^{-2}$ and equilibrium $\mu(|p|^2)$, so \cite[Theorems~1.1 and~2.6]{NY1} apply verbatim and their quantitative statements transfer through $\rho_\mu=\int\mu=(2\pi^2)^{-1}\int g=4\pi\rho_g$; e.g.\ $\tau_*(0)=\sqrt{2\rho_\mu}$ becomes $\omega_p=\sqrt{8\pi\rho_g}$. (\cite[Remark~2.8]{NY1} covers general long-range kernels with $\widehat w\ge0$, $\widehat w(0)=\infty$, $\lim_{|k|\to\infty}\widehat w<\infty$, $\widehat w'\le0$; the rescaling above is sharper in that it imports the quantitative statements as well.)

\section{Uniform choice of the Rouch\'e disks}\label{app:disks}

First take $r_0$ smaller than $\tfrac14\min_{r\in I}\{\tau_0(r)-(2r\Upsilon+r^2)\}$ and $\tfrac14\min_{r\in I}2\tau_0(r)$; then for every $k\in\Omega_I$ the closed disks $\overline{U_{\pm,k}}$ stay within the gap region of Lemma~\ref{lem:gap} and are disjoint.

We claim that, after further shrinking $r_0$, each disk contains exactly one zero of $D_0(\cdot,k)$ (namely $\sigma i\tau_0(|k|)$, simple), \emph{uniformly} in $k\in\Omega_I$. Suppose not. Then there are radii $r_n\downarrow0$, wave numbers $k_n\in\Omega_I$, a sign $\sigma$, and zeros $\lambda_n\ne\sigma i\tau_0(|k_n|)$ of $D_0(\cdot,k_n)$ (or zeros of multiplicity $\ge2$) with $|\lambda_n-\sigma i\tau_0(|k_n|)|<r_n$. Passing to a subsequence, $k_n\to k_*\in\Omega_I$ and $\lambda_n\to\lambda_*:=\sigma i\tau_0(|k_*|)$. Since $\lambda_*$ is a simple isolated zero of $D_0(\cdot,k_*)$ (Assumption~\ref{ass:input}), fix $\rho>0$ so small that the closed disk $\overline{B(\lambda_*,\rho)}$ lies in the gap region and contains no zero of $D_0(\cdot,k_*)$ other than $\lambda_*$; then
\[
m:=\min_{|\lambda-\lambda_*|=\rho}\bigl|D_0(\lambda,k_*)\bigr|>0 .
\]
By continuity in $k$ of the integrand in \eqref{eq:D0} on the gap region (Lemma~\ref{lem:gap}) and dominated convergence, $D_0(\cdot,k_n)\to D_0(\cdot,k_*)$ uniformly on $\overline{B(\lambda_*,\rho)}$; hence for all large $n$,
\[
\bigl|D_0(\lambda,k_n)-D_0(\lambda,k_*)\bigr|<m\le\bigl|D_0(\lambda,k_*)\bigr|
\qquad\text{on the circle }|\lambda-\lambda_*|=\rho,
\]
and Rouch\'e's theorem gives that $D_0(\cdot,k_n)$ has exactly as many zeros in $B(\lambda_*,\rho)$, counted with multiplicity, as $D_0(\cdot,k_*)$ --- namely one. But for large $n$ the disk $B(\lambda_*,\rho)$ contains both $\sigma i\tau_0(|k_n|)$ and the extra zero $\lambda_n$ (or a multiple zero), giving a count $\ge2$ --- a contradiction.

Finally, $d_0>0$ in \eqref{eq:d0}: the map $(k,\theta,\sigma)\mapsto|D_0(\sigma i\tau_0(|k|)+r_0e^{i\theta},k)|$ is continuous and positive on the compact set $\Omega_I\times[0,2\pi]\times\{\pm\}$, hence attains a positive minimum.

\end{document}